\documentclass[11pt,english]{article}

\usepackage[margin = 2.5cm]{geometry}

\usepackage{amsthm, tikz}
\usepackage{amsmath}
\usepackage{amssymb}
\usepackage{mathtools}
\usepackage{graphicx}
\usepackage[pdftex,colorlinks,backref=page,citecolor=blue,bookmarks=false]{hyperref}
\usepackage{color}
\usepackage{enumitem}
\usepackage{thm-restate}

\usepackage{caption}
\usepackage[square,sort,comma,numbers]{natbib} 
\usepackage{setspace}
\usepackage{cleveref}
\theoremstyle{plain}

\newtheorem*{theorem*}{Theorem}
\newtheorem{theorem}{Theorem}[section]
\crefname{theorem}{Theorem}{Theorems}
\Crefname{theorem}{Theorem}{Theorems}

\newtheorem*{lemma*}{Lemma}
\newtheorem{lemma}[theorem]{Lemma}
\crefname{lemma}{Lemma}{Lemmas}
\Crefname{lemma}{Lemma}{Lemmas}

\newtheorem*{claim*}{Claim}
\newtheorem{claim}[theorem]{Claim}
\crefname{claim}{Claim}{Claims}
\Crefname{claim}{Claim}{Claims}

\newtheorem{proposition}[theorem]{Proposition}
\crefname{proposition}{Proposition}{Propositions}
\Crefname{proposition}{Proposition}{Propositions}

\crefname{corollary}{Corollary}{Corollaries}
\Crefname{corollary}{Corollary}{Corollaries}

\newtheorem{conjecture}[theorem]{Conjecture}
\crefname{conjecture}{Conjecture}{Conjectures}
\Crefname{conjecture}{Conjecture}{Conjectures}

\crefname{question}{Question}{Questions}
\Crefname{question}{Question}{Questions}

\crefname{observation}{Observation}{Observations}
\Crefname{observation}{Observation}{Observations}

\newtheorem{example}[theorem]{Example}
\crefname{example}{Example}{Examples}
\Crefname{example}{Example}{Examples}

\theoremstyle{definition}

\crefname{problem}{Problem}{Problems}
\Crefname{problem}{Problem}{Problems}

\newtheorem{definition}[theorem]{Definition}
\crefname{definition}{Definition}{Definitions}
\Crefname{definition}{Definition}{Definitions}

\theoremstyle{remark}
\newtheorem*{remark}{Remark}
\crefname{remark}{Remark}{Remarks}
\Crefname{remark}{Remark}{Remarks}

\usepackage{xpatch}
\xpatchcmd{\proof}{\itshape}{\normalfont\proofnamefont}{}{}
\newcommand{\proofnamefont}{}
\renewcommand{\proofnamefont}{\bfseries}

\usepackage{framed}

\newcommand{\remove}[1]{}

\newcommand{\floor}[1]{
	\left\lfloor #1 \right\rfloor
}

\newcommand{\modthree}[1]{\equiv #1 \,(\mathrm{mod} \,\,3)}

\def\eps{\varepsilon}

\def \HH {\mathcal{H}}
\def \C {\mathcal{C}}

\def \F {\mathcal{F}}
\def \Fc {\mathcal{F}}
\def \Fp {\mathcal{F}^+}
\def \G {\mathcal{G}}
\def \be {\bar{e}}
\def \As {A^*}
\def \Bs {B^*}
\def \Hs {\mathcal{H}^*}
\def \tP {\tilde{P}}
\def \tQ {\tilde{Q}}

\DeclareMathOperator{\final}{final}
\DeclareMathOperator{\inn}{in}
\DeclareMathOperator{\out}{out}
\DeclareMathOperator{\ex}{ex}
\DeclareTextCompositeCommand{\v}{OT1}{l}{l\nobreak\hspace{-.1em}'}

\title{The Tur\'an density of tight cycles in three-uniform hypergraphs}

\author{
	Nina Kam\v{c}ev\thanks{Department of Mathematics, Faculty of Science, University of Zagreb, Croatia. Email: \texttt{nina.kamcev@math.hr}. Research supported by the European Union’s Horizon 2020 research and innovation programme [MSCA GA No 101038085].}
	\and
	Shoham Letzter\thanks{
		Department of Mathematics, 
		University College London, 
		Gower Street, London WC1E~6BT, UK. 
		Email: \texttt{s.letzter}@\texttt{ucl.ac.uk}. 
		Research supported by the Royal Society.
	}
	\and
	Alexey Pokrovskiy\thanks{	Department of Mathematics, 
		University College London, 
		Gower Street, London WC1E~6BT, UK. 
		Email: \texttt{dralexeypokrovskiy@gmail.com}.}
}

\begin{document}
	
	\date{}
	\maketitle
	
	\begin{abstract}
		
		\setlength{\parskip}{\medskipamount}
		\setlength{\parindent}{0pt}
		\noindent
		The \emph{Tur\'an density} of an $r$-uniform hypergraph $\HH$, denoted $\pi(\HH)$, is the limit of the maximum density of an $n$-vertex $r$-uniform hypergraph not containing a copy of $\HH$, as $n \to \infty$. 
		
		Denote by $\C_{\ell}$ the $3$-uniform tight cycle on $\ell$ vertices. Mubayi and R\"odl gave an ``iterated blow-up'' construction showing that the Tur\'an density of $\C_5$ is  at least $2\sqrt{3} - 3 \approx 0.464$, and this bound is conjectured to be tight. 
		Their construction also does not contain $\C_{\ell}$ for larger $\ell$ not divisible by $3$, which suggests that it might be the extremal construction for these hypergraphs as well.
		Here, we determine the Tur\'an density of $\C_{\ell}$ for all large $\ell$ not divisible by $3$, showing that indeed $\pi(\C_{\ell}) = 2\sqrt{3} - 3$.        To our knowledge, this is the first example of a Tur\'an density being determined where the extremal construction is an iterated blow-up construction.
		
		A key component in our proof, which may be of independent interest, is a $3$-uniform analogue of the statement ``a graph is bipartite if and only if it does not contain an odd cycle''.
	\end{abstract}
	
	\section{Introduction}
	For an $r$-uniform hypergraph $\mathcal H$, the \emph{Tur\'an number} of $\mathcal H$, denoted $\ex(n, \mathcal H)$, is defined as the maximum number of edges an $n$-vertex $r$-uniform hypergraph can have without containing a copy of $\mathcal H$ as a subgraph. For ($2$-uniform) graphs, we have a fairly good understanding of Tur\'an numbers.
	The first theorem proved about them is Mantel's theorem \cite{mantel1907problem}, which says that, for the triangle, we have $\ex(n, K_3)= \floor{n^2/4}$. This was generalised by Tur\'an \cite{turan1941extremal} who showed that $\ex(n,K_r) \approx (1-\frac{1}{r-1})\binom{n}{2}$. 
	For non-complete graphs we know less, and usually only know what the Tur\'an number of a graph is asymptotically, up to $o(n^2)$ terms.
	Because of this, we study the \emph{Tur\'an density} of an $r$-uniform hypergraph $\mathcal H$, denoted $\pi(\HH)$, and defined as $\pi(\mathcal H)=\lim_{n\to \infty} \frac{\ex(n,\mathcal H)}{\binom n r}$. This limit is known to exist, and, moreover, it is clear that $\pi(\mathcal H)\in [0,1]$ for every $H$. 
	The Tur\'an densities of ($2$-uniform) graphs were completely determined by Erd\H{o}s and Stone~\cite{erdos1946structure}, who  showed that every graph $H$ satisfies $\pi(H)=1-\frac{1}{\chi(H)-1}$.

	The special case of the Erd\H{o}s--Stone theorem for bipartite graphs can be generalised to higher uniformities, as follows (see \cite{erdos1971extremal}): every $r$-partite $r$-uniform hypergraph $\HH$ satisfies $\pi(\HH) = 0$. (An $r$-uniform hypergraph $\HH$ is said to be \emph{$r$-partite} if its vertices can be $r$-coloured so that every edge has one vertex of each colour.)
	Nevertheless, in general, our understanding of Tur\'an numbers in higher uniformities is very limited, and there are only a small number of hypergraphs whose Tur\'an densities are known; see Keevash \cite{keevash11} for a comprehensive survey of the topic listing a number of such hypergraphs. 
	A notable, relatively early example is a result of de Caen and F\"uredi \cite{de2000maximum} showing that the Tur\'an density of the Fano plane is  $3/4$ (see also \cite{furedi2005triple, keevash2005turan}). More recently,  the impactful computer-assisted ``flag-algebra'' technique has been used to obtain a number of sharpest known upper bounds on Tur\'an densities (see \cite{baber2011new, keevash11, razborov2013flag} and the references therein). 
	
	Given the sporadicity of hypergraphs whose Tur\'an densities are known, it is unsurprising that there are many conjectures about Tur\'an densities of specific hypergraphs. The most famous of these is Tur\'an's conjecture \cite{turan1961research}, that the Tur\'an density of the \textit{tetrahedron} $K_4^{(3)}$ is $5/9$. Frankl and F\"uredi \cite{ff84} conjectured that the  Tur\'an density of the 3-edge subgraph of $K_4^{(3)}$ (usually denoted $K_4^-$) is $2/7$. 
	A particularly relevant conjecture for us concerns tight   cycles. The \emph{$r$-uniform tight cycle} of length $\ell$, denoted $\mathcal C_{\ell}^r$, is defined to be the hypergraph with vertex set $\{1, \dots, \ell\}$ and hyperedges all sets of the form $\{x, x+1,\dots,   x+r-1 \!\pmod {\ell}\}$. The following conjecture, usually attributed to Mubayi and R\"odl, appears for instance in \cite{falgas2012turan,mubayi2011hypergraph}. 
	\begin{conjecture}
		\label{Conjecture_Mubayi_Rodl}
		$\pi(\mathcal C_5^3)=2\sqrt 3-3$.
	\end{conjecture}
	The lower bound $\pi(\mathcal C_5^3) \geq 2\sqrt 3-3 \approx 0.464$ was found by Mubayi and R\"odl (see \Cref{ex:iterated-blow-up} below for a description of their example) and the best upper bound is due to Razborov~\cite{razborov20103}, who showed $\pi(\mathcal C_5^3)\leq 0.468$. 
	
	One basic reason why hypergraphs are more difficult than graphs is that the extremal $\mathcal H$-free hypergraphs  can be much more complicated than the extremal graphs. In the 2-uniform case, the Erd\H{o}s--Stone theorem shows that all optimal graphs are close to being complete multipartite. For higher-uniformity hypergraphs there have been numerous papers discovering more complicated possible extremal hypergraphs, for instance~\cite{brown83,ff84,blm11}, as well as Conjectures~\ref{Conjecture_Mubayi_Rodl} and~\ref{conj:c5-minus}. For some hypergraphs, such as $K_4^{(3)}$, the conjectured extremal constructions are even non-unique and very different from each other~\cite{brown83,kostochka84,flaass88,razborov11}.  
	
	One class of extremal examples, which does not occur for graphs, is an ``iterated blow-up construction''. The conjectured extremal example for \Cref{Conjecture_Mubayi_Rodl} is an instance of such a construction.
	\begin{example}[Iterated blow-up construction with no copies of $\C_{5}^3$] \label{ex:iterated-blow-up}
		Consider nested vertex sets $V_1\supseteq \ldots \supseteq  V_{t}$  with $|V_i| - |V_{i+1}| = x_i$ for $i \in [t]$, with the convention $|V_{t+1}| = \emptyset$. Let $\mathcal H(x_1, \ldots, x_t)$ be a 3-uniform hypergraph on the vertex set $V_1$, where $xyz$ is an edge whenever $x,y\in V_i \setminus V_{i+1}$ and $z\in V_{i+1}$ for some $i$ (see Figure~\ref{Figure_extremal_hypergraph}).
		
		We claim that there is no copy of $\C_5^3$. To see this, say that an edge with two vertices in $V_i \setminus V_{i+1}$ and one vertex in $V_{i+1}$ has \emph{type $i$}, and observe that if two edges $e$ and $f$ intersect in two vertices, they are of the same type. Thus, if $C = (u_1 \ldots u_5)$ is a cycle, then its edges all have the same type, say $i$. Without loss of generality, $u_1, u_2 \in V_i \setminus V_{i+1}$ and  $u_3 \in V_{i+1}$. It follows that $u_4 \in V_i \setminus V_{i+1}$, and thus $u_5 \in V_i \setminus V_{i+1}$. But then $u_4 u_5 u_1$ is not an edge of $\HH(x_1, \ldots, x_t)$, a contradiction.
		
		Thus $\pi(\mathcal C_5^3)\geq e(\mathcal H(x_1, \ldots, x_t))/\binom n3$ for all choices of $x_1, \ldots, x_t$ with $x_1 + \ldots + x_t = n$. Let $f(n)$ denote the maximum number of edges that such a hypergraph on $n$ vertices can have i.e $f(n):=\max(e(\mathcal H(x_1, \ldots, x_t): x_1, \dots, x_t \geq 1,\, x_1 + \dots + x_t = n)$. 
		It is possible to show $\lim_{n\to \infty} f(n)/\binom n3=2\sqrt 3-3$ (see Section~\ref{sec:optimal} for details), which gives $\pi(\mathcal C_5^3)\geq 2\sqrt 3-3$.
		
		Let $\G_n = \HH(x_1, \ldots, x_t)$ for a choice of $x_1 \ge \ldots \ge x_t$ such that $n = x_1 + \ldots + x_t$ and $e(\G_n) = f(n)$.
	\end{example}
	
	\begin{figure}
		\tikzset{every picture/.style={line width=0.75pt}} 
		
		\begin{center}
			\begin{tikzpicture}[x=0.75pt,y=0.75pt,yscale=-1,xscale=1]
				
				\draw  [draw opacity=0][fill={rgb, 255:red, 180; green, 180; blue, 180 }  ,fill opacity=1 ] (407.01,155.57) -- (266,182.14) -- (266,129) -- cycle ;
				\draw   (100,53.15) .. controls (100,40.99) and (109.85,31.14) .. (122.01,31.14) -- (463,31.14) .. controls (475.15,31.14) and (485.01,40.99) .. (485.01,53.15) -- (485.01,246.14) .. controls (485.01,258.29) and (475.15,268.14) .. (463,268.14) -- (122.01,268.14) .. controls (109.85,268.14) and (100,258.29) .. (100,246.14) -- cycle ;
				\draw   (228,58.94) .. controls (228,47.45) and (237.31,38.14) .. (248.8,38.14) -- (453.21,38.14) .. controls (464.69,38.14) and (474.01,47.45) .. (474.01,58.94) -- (474.01,241.34) .. controls (474.01,252.83) and (464.69,262.14) .. (453.21,262.14) -- (248.8,262.14) .. controls (237.31,262.14) and (228,252.83) .. (228,241.34) -- cycle ;
				\draw   (310.01,59.81) .. controls (310.01,51.71) and (316.57,45.14) .. (324.68,45.14) -- (453.34,45.14) .. controls (461.44,45.14) and (468.01,51.71) .. (468.01,59.81) -- (468.01,240.47) .. controls (468.01,248.57) and (461.44,255.14) .. (453.34,255.14) -- (324.68,255.14) .. controls (316.57,255.14) and (310.01,248.57) .. (310.01,240.47) -- cycle ;
				\draw   (382.01,58.29) .. controls (382.01,54.34) and (385.21,51.14) .. (389.15,51.14) -- (451.86,51.14) .. controls (455.8,51.14) and (459.01,54.34) .. (459.01,58.29) -- (459.01,229.99) .. controls (459.01,233.94) and (455.8,237.14) .. (451.86,237.14) -- (389.15,237.14) .. controls (385.21,237.14) and (382.01,233.94) .. (382.01,229.99) -- cycle ;
				\draw  [draw opacity=0][fill={rgb, 255:red, 176; green, 176; blue, 176 }  ,fill opacity=1 ] (405.01,94.82) -- (180.5,134.14) -- (180.5,55.5) -- cycle ;
				\draw  [fill={rgb, 255:red, 0; green, 0; blue, 0 }  ,fill opacity=1 ] (343.01,155.07) -- (248,181.14) -- (248,129) -- cycle ;
				\draw  [color={rgb, 255:red, 0; green, 0; blue, 0 }  ,draw opacity=1 ][fill={rgb, 255:red, 0; green, 0; blue, 0 }  ,fill opacity=1 ] (405.01,206.07) -- (363,225.14) -- (363,187) -- cycle ;
				\draw  [draw opacity=0][fill={rgb, 255:red, 61; green, 61; blue, 61 }  ,fill opacity=1 ] (333.01,93.07) -- (154,131.14) -- (154,55) -- cycle ;
				\draw  [color={rgb, 255:red, 0; green, 0; blue, 0 }  ,draw opacity=1 ][fill={rgb, 255:red, 0; green, 0; blue, 0 }  ,fill opacity=1 ] (263.01,91.57) -- (137,129.14) -- (137,54) -- cycle ;
				
				\draw (131,193.4) node [anchor=north west][inner sep=0.75pt]  [font=\LARGE]  {$V_{1}$};
				\draw (251,194.4) node [anchor=north west][inner sep=0.75pt]  [font=\LARGE]  {$V_{2}$};
				\draw (318,197.4) node [anchor=north west][inner sep=0.75pt]  [font=\LARGE]  {$V_{3}$};
				\draw (419,195.4) node [anchor=north west][inner sep=0.75pt]  [font=\LARGE]  {$V_{4}$};
			\end{tikzpicture}
		\end{center}
		\caption{An illustration of the hypergraph $\mathcal H(x_1, x_2, x_3, x_4)$. }\label{Figure_extremal_hypergraph}
	\end{figure}
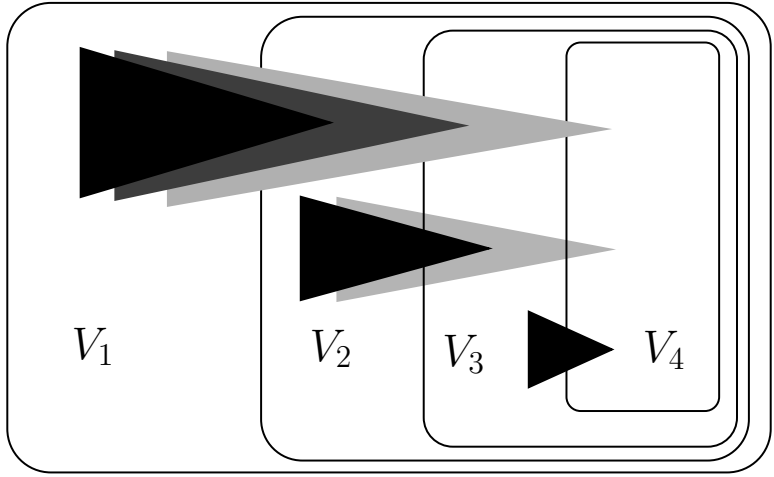
	
	Note that in the above construction, $\mathcal G_n$ has no tight cycles of lengths $\ell\equiv 1$ or $2 \!\pmod 3$ either. So it is plausible that Conjecture~\ref{Conjecture_Mubayi_Rodl} could be strengthened to say that $\pi(\mathcal C_\ell^3)=2\sqrt 3-3$ for all $\ell\geq 5$ with $\ell\equiv 1$ or $2 \pmod 3$ (notice that $\C_4^3 = K_4^{(3)}$, and there are known examples of $K_4^{(3)}$-free $3$-uniform graphs with density at least $5/9 > 2\sqrt{3} - 3$). The main result of our paper is to show that this is true for sufficiently large $\ell$.
	\begin{restatable}{theorem}{thmSingleCycle}\label{thm:single-cycle}
		Let $\ell$ be sufficiently large with $\ell\equiv 1$ or $2 \!\pmod 3$. Then $\pi(\mathcal C_\ell^3)=2\sqrt 3-3$.
	\end{restatable}
	To our knowledge, this is the first example of a Tur\'an density being determined where the extremal construction is an iterated blow-up construction, and could be a step towards Conjecture~\ref{Conjecture_Mubayi_Rodl}.
	This is also one of the few examples of hypergraphs with irrational Tur\'an densities. Such hypergraphs were recently found by Yan and Peng~\cite{yp22}, as well as Wu~\cite{wu22}, motivated by the work of Chung and Graham~\cite{cg98}, Baber and Talbot~\cite{baber2011new}, and Pikhurko~\cite{pikhurko2014possible}.
	We remark that Conjecture~\ref{Conjecture_Mubayi_Rodl} would imply Theorem~\ref{thm:single-cycle} via \Cref{theorem_blow_up} below (using the same argument as in the proof of Theorem~\ref{thm:single-cycle} in Section~\ref{sec:diameter}).
	
	One of our main tools, which may be of independent interest, is a 3-uniform analogue of the statement ``a graph is bipartite if and only if it does not contain an odd cycle''; see \Cref{thm:good-colouring}. Thus, we characterise 3-uniform hypergraphs $\HH$ which do not contain \textit{homomorphic images} of cycles $\C^{3}_\ell$ with $3 \nmid \ell$, in terms of certain colourings of $V(\HH)^2$, as explained in the proof overview.
	
	Throughout the paper we will informally refer to 3-uniform cycles of length $\ell\equiv 1$ or $2 \pmod 3$ as \emph{odd cycles}, and we will often refer to $3$-uniform hypergraphs as 3-graphs.
	
	\subsection*{Related results}
	As we mentioned, there are very few hypergraphs with a known Tur\'an density, but let us state some recent results on Tur\'an-type problems for tight cycles.
	A well-studied hypergraph parameter is the so-called \emph{uniform Tur\'an density}, the infimum over all $d$ for which any
	sufficiently large hypergraph with the property that all its linear-size subhypergraphs have density at
	least $d$ contains $\HH$. This line of research was initiated by Erd\H{o}s and S\'os 
	\cite{erdHos1982ramsey} and, parallel to the classical Tur\'an densities, the motivating questions in the area are determining the uniform Tur\'an densities of the tetrahedron $K_4^{(3)}$ and its 3-edge subgraph $K_4^-$. The latter was found to be $1/4$ by Glebov, Kr\'a\v{l}, and Volec~\cite{gkv16} and  later by Reiher, R\"odl, and Schacht~\cite{rrs18} with a different proof. In 2022, Buci\'c, Cooper, Kr{\'a}\v{l}, Mohr, and Munha Correia showed that for $\ell \geq 5$ and not divisible by 3, the uniform Tur\'an density of $\C_{\ell}^3$ is $\frac{4}{27}$~\cite{bckmm21}. 
	
	Another question that has attracted a lot of interest in the last few years is, what is the extremal number of tight cycles (the maximum number of edges in an $n$-vertex $r$-uniform hypergraph containing no tight cycles)? For $r=2$, the answer is of course $n-1$, but it turns out that the behaviour is rather different for $r\geq 3$. More specifically, after a series of results~\cite{hm19,janzer21,st22,letzter2021hypergraphs}, we know that the extremal number of tight $r$-uniform cycles lies between $\Omega\left(n^{r-1} \log n / \log \log n\right)$ and $O\left(n^{r-1} \log^5 n \right)$.
	
	\section{Proof overview} \label{sec:overview}
	For an $r$-uniform hypergraph $\HH$, the \emph{$t$-blow-up} of $\HH$, denoted $\HH[t]$, is defined to be the $r$-uniform hypergraph with vertex set $V(\HH)\times [t]$ and edges all $r$-tuples $\{(x_1, i_1), \dots, (x_r, i_r)\}$ with $\{x_1, \dots, x_r\}\in E(\HH)$. 
	The starting point of our proof is the following theorem, which asserts that the blow-up of a hypergraph $\HH$ has the same Tur\'an density as $\HH$. 
	\begin{theorem}[\cite{keevash11}, Theorem 2.2] \label{theorem_blow_up}
		Let $t$ be an integer and let $\HH$ be an $r$-uniform hypergraph. Then $\pi(\HH[t]) = \pi(\HH)$.
	\end{theorem}
	It shows that, rather than focusing on the Tur\'an density $\pi(\mathcal C_k^3)$ for an odd cycle $C_k^3$, we can instead work out the Tur\'an density of $\pi(\HH)$ for any hypergraph $\HH$ whose blow-up $\HH[t]$ contains $\mathcal C_k^3$ for some $t$. We refer to such hypergraphs $\HH$ as pseudocycles, and they can be equivalently defined as follows.
	\begin{definition} \label{def:pseudocycle}
		A \emph{pseudocycle} of length $\ell$ in a 3-uniform hypergraph $\HH$ is a sequence of (not necessarily distinct) vertices $v_1, \ldots, v_{\ell}$, such that for each $i \in [\ell]$, we have that $\{v_i, v_{i+1\!\pmod{\ell}}, v_{i+2\!\pmod{\ell}}\}$ is an edge of $\HH$.
		A \emph{pseudopath} of order $\ell$ is defined analogously.
	\end{definition}
	It is easy to show that for a hypergraph $\HH$, the properties ``$\HH[t]$ contains a $\mathcal C_k^3$ for some $t$'' and ``$\HH$ contains a length $k$ pseudocycle'' are equivalent.
	
	Thus, the starting point of our approach is, what is the maximum number of edges that a 3-uniform hypergraph can have without containing an odd pseudocycle? Later (after Corollary~\ref{corollary_intro_long_pseudocycles}), we will discuss how to forbid only \textit{short} pseudocycles. To understand our approach to this question, consider the analogous question about graphs --- what is the maximum number of edges in a (2-uniform) graph with no odd circuits? By Kotzig's Lemma, a graph has no odd circuit if, and only if, it is bipartite. Thus, the maximum number of edges in an $n$-vertex bipartite graph is $\floor{\frac{n^2}{4}}$.
	
	Our approach to the 3-uniform case is analogous to this. We first find the relevant generalisation of bipartite graphs, and then maximise the number of edges over this class of graphs. To define this generalisation, recall	 
	that a graph is bipartite if, and only if, it has a proper 2-vertex-colouring. In our context, we will be colouring the shadow of a 3-uniform hypergraph.   The \emph{shadow} of a hypergraph $\HH$, denoted $\partial \HH$, is the graph on vertices $V(\HH)$ whose edges are pairs $xy$ that are contained in an edge in $\HH$. 
	\begin{definition} \label{def:good-col}
		A \emph{good colouring} of a 3-uniform hypergraph $\HH$ is a colouring of its shadow, such that each edge $xy$ in the shadow is either coloured blue or coloured red and given an orientation, and every edge $e$ in $\HH$ can be written as $xyz$ where $xy$ and $xz$ are red and directed from $x$ and $yz$ is blue.
	\end{definition}
	The key first step of our proof is to show that the notion of ``good colouring'' is exactly equivalent to $\HH$ not containing an odd pseudocycle.
	\begin{restatable}{theorem}{thmGoodColouring}\label{thm:good-colouring}
		A 3-uniform hypergraph $\HH$ has a good colouring if, and only if, $\HH$ has no pseudocycle of length $\ell$ with $3 \nmid \ell$.
	\end{restatable}
	\def \mcherry {m_{\mathrm{cherry}}}
	\def \mgood {m_{\mathrm{good}}}
	This theorem is proved in Section~\ref{sec:good-col}.
	Having established the above theorem, we next wish to maximise the  number of edges in a hypergraph with a good colouring. To this end, we define a \emph{coloured graph} to be a complete graph whose edges are either coloured blue or coloured red and oriented.
	A \emph{cherry} in a coloured graph $G$ is a triple $xyz$ such that $xy$ and $xz$ are red and directed from $x$ and $yz$ is blue. Denote by $c(G)$ the number of cherries in $G$. 
	Notice that if we have a good colouring of the shadow of $\HH$ (and the remaining vertex pairs can be coloured arbitrarily), then all edges of $\HH$ will be cherries in the resulting coloured graph. Thus, let $\mcherry(n)$ be the maximum number of cherries in an $n$-vertex coloured graph. The quantity $\mcherry(n)$ has been studied before  by Falgas-Ravry and Vaughan \cite{falgas2012turan}, who used flag algebras to show that $\lim_{n\to \infty} \mcherry/\binom n3=2\sqrt 3-3$.  Huang \cite{huang2014maximum} worked on the area further and determined the maximum number of ``induced out-stars'' of size $t$ in an $n$-vertex coloured graph. We remark that the afore-mentioned authors used an equivalent reformulation -- they studied the maximum number of ``induced out-stars" in an uncoloured directed graph, and there is a clear correspondence between a coloured graph (in our sense) and a directed graph. The following special case of their results is relevant for us. 
	\begin{theorem}[Falgas-Ravry--Vaughan~\cite{falgas2012turan}; Huang~\cite{huang2014maximum}]
	\label{theorem_intro_falgas_ravry}
		Every coloured graph on $n$ vertices contains at most $f(n)$ cherries. 
	\end{theorem}
	Combining Theorem~\ref{thm:good-colouring} and \Cref{theorem_intro_falgas_ravry} already yields the following weakening of Theorem~\ref{thm:single-cycle}.
	\begin{restatable}{corollary}{corLongPseudocylces} \label{corollary_intro_long_pseudocycles}
		If $\HH$ is a $3$-uniform hypergraph on $n$ vertices which  does not contain a pseudocycle of length $\ell$ for any $\ell$ with $3 \nmid \ell$, then $e(\HH) \le f(n)$.
	\end{restatable}
	Notice that in Theorem~\ref{thm:single-cycle} we forbid odd pseudocycles of a single length, whereas in Corollary~\ref{corollary_intro_long_pseudocycles} odd pseudocycles of \emph{all} lengths are forbidden. 
	Thus, the next goal is to prove a version of the above corollary which holds when forbidding \emph{short} odd pseudocycles, yielding a finite family of forbidden hypergraphs. This is done by controlling the diameter of the hypergraph $\HH$. 
	
	\begin{definition}
		The \emph{diameter} of a hypergraph $\HH$ is the minimum $\ell$ such that the following holds: for every $x, y, z, w \in V(\HH)$ (where $x, y$ are distinct and $z, w$ are distinct) whenever there is a pseudopath from $xy$ to $zw$, there is such a pseudopath of order at most $\ell$.
	\end{definition} 
	In Section~\ref{sec:diameter}, we show that, for $\ell\gg \epsilon^{-1}$, every 3-uniform hypergraph $\HH$ contains a subhypergraph $\HH'$ with $e(\HH')\geq e(\HH)-\epsilon n^3$ such that $\HH'$ has  diameter at most $\ell$ (see Proposition~\ref{prop:small-diameter}). Then we show that in every 3-uniform hypergraph of diameter $\ell$, if there is some odd pseudocycle, then there is also an odd pseudocycle of length at most $4\ell$ (see Proposition~\ref{prop:diam-cyc}). Combining these with Corollary~\ref{corollary_intro_long_pseudocycles} shows the following 
	\begin{restatable}{corollary}{corWeak} \label{cor:weak-main}
		Let $1/n \ll 1/\ell \ll \epsilon \ll 1$, and let $\HH$ be an $n$-vertex hypergraph with no odd pseudocycles of length at most $\ell$. Then $e(\HH) \le f(n) + \epsilon n^3$.
	\end{restatable}
	
	Note that this is still not strong enough to combine with Theorem~\ref{theorem_blow_up} to yield Theorem~\ref{thm:single-cycle}. The issue is that the length of the cycle $\ell$ depends on $\epsilon$ --- therefore, when combined with Theorem~\ref{theorem_blow_up}, we would only get that $\lim_{m\to \infty}\pi(\mathcal C_m^3)=2\sqrt 3-3$. To go further, we prove a ``stability version'' of Theorem~\ref{theorem_intro_falgas_ravry}. We show that if a coloured graph $D$ on $n$ vertices contains more than $f(n)-\epsilon n^3$ cherries, then $D$ must have a very constrained structure similar to the iterated blow-up construction (see Theorem~\ref{thm:stability} for the precise statement). Once we have this, we can obtain the following strengthening of Corollary~\ref{cor:weak-main}
	\begin{restatable}{theorem}{thmPseudocycle} \label{thm:pseudocycles}
		There exists $L > 0$ such that the following holds. If $\HH$ is a $3$-uniform hypergraph on $n$ vertices  which  does not contain a pseudocycle of length $\ell$ for any $\ell \leq L$ with $3 \nmid \ell$, then $e(\HH) \le f(n) + O(1)$.
	\end{restatable}
	This theorem easily combines with Theorem~\ref{theorem_blow_up} in order to give our main result, Theorem~\ref{thm:single-cycle} (see \Cref{sec:diameter}).

	\section{Finding a good colouring} \label{sec:good-col}
	Recall that a \emph{pseudocycle} of order $m$ (or \emph{$m$-pseudocycle}) is a sequence $v_1 \ldots v_m$ of not necessarily distinct vertices such that $v_i v_{i+1} v_{i+2}$ is an edge for $i \in [m]$ (indices taken mod $3$). A \emph{pseudopath} of order $m$ is defined analogously.
	A hypergraph is called \textit{tightly connected} if there is a pseudopath between any two edges. Given vertices $x, y, z, w$ (not necessarily distinct), a pseudopath from $xy$ to $zw$ (where $xy$ and $zw$ are ordered pairs) is a pseudopath whose first two vertices are $x$ and $y$ (in this order) and the last two vertices are $z$ and $w$. 
	The \emph{shadow} of a hypergraph $\HH$, denoted $\partial \HH$, is the graph on vertices $V(\HH)$ whose edges are pairs $xy$ that are contained in an edge in $\HH$.
	
	Recall that a \emph{good colouring} of a hypergraph $\HH$ is a colouring of its shadow, such that each edge $xy$ in the shadow is either coloured blue or coloured red and given an orientation, and every edge $e$ in $\HH$ can be written as $xyz$ where $xy$ and $xz$ are red and directed from $x$ and $yz$ is blue. Such an edge is called a \emph{cherry}, and the  vertex $x$ is called its \emph{apex}. 
	\begin{figure}[h!]
		\centering
		\includegraphics[scale = 1]{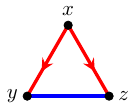}
		\vspace{-.3cm}
		\caption{A cherry $xyz$ with apex $x$}
	\end{figure}
	
	In this section, we will prove \Cref{thm:good-colouring}, restated here. 
	\thmGoodColouring*
	It is easy to see that a hypergraph with a good colouring has no pseudocycles of length $\ell$ with $3 \nmid \ell$, so the main effort will be put into proving the ``if'' direction.
	Namely, we need to show that every hypergraph with no odd pseudocycles has a good colouring. Before specifying such a colouring, let us give some intuition. Any proper path (that is, with no repetitions) $v_1 \ldots v_k$ has a good colouring, and this colouring is unique given the colour of $v_1 v_2$ (see \Cref{fig:path} for the three good colourings of a path of order 9, and notice that each such colouring colours each edge in the shadow differently). 
	\begin{figure}[h!]
		\centering
		\includegraphics[scale = .9]{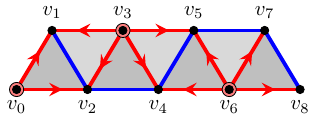}
		\hspace{.4cm}
		\includegraphics[scale = .9]{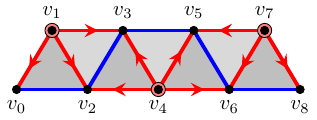}
		\hspace{.4cm}
		\includegraphics[scale = .9]{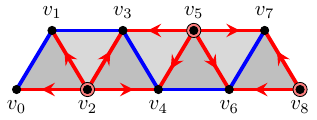}
		\caption{The three good colourings of a path of order $9$}
		\label{fig:path}
	\end{figure}
	
	A proper cycle has a good colouring if and only if it is tripartite (i.e.~the number of vertices is divisible by 3). See \Cref{fig:cycle} for a good colouring of a cycle of length 18 and notice that every third vertex is an apex.
	\begin{figure}[h!]
		\centering
		\includegraphics[scale = 1]{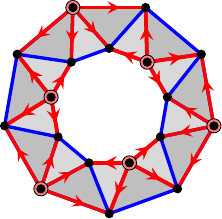}
		\caption{A good colouring of a cycle whose length is divisible by $3$}
		\label{fig:cycle}
	\end{figure}
	
	Moreover, if there is a path $P = xy \ldots yx$, then the order of $P$ uniquely determines the colour of $xy$. This fact will be used to construct a good colouring of our hypergraph $\HH$ -- we will start from a specific pair $xy$ and extend the colouring uniquely along pseudopaths. The difficulty is to show that this colouring is well defined, so the actual colouring definition will involve some more formalism.
	For a pseudopath $P = v_1 \ldots v_k$, define $\tilde{P}$ by
	\begin{equation} \label{eqn:tilde}
		\tP := v_{k-1} v_k v_{k-2} v_{k-1} v_{k-3} \dots v_4 v_2 v_3 v_1 v_2; 
	\end{equation} 
	note that $\tP$ is a pseudopath from $v_{k-1}v_k$ to $v_1 v_2$ of order $2k-2$ (because every vertex but $v_1$ and $v_k$ appears twice).

	\begin{proof}[Proof of \Cref{thm:good-colouring}]
		
		Whenever we talk about a path or cycle in this proof, we mean a pseudopath or pseudocycle. 
		
		As we said above, it is easy to show that a hypergraph with a good colouring has no odd cycles, so it suffices to show that if $\HH$ has no odd cycles then $\HH$ has a good colouring.
		Note that we may assume that $\HH$ is tightly connected, by adding edges if necessary.
		
		Let $P_0$ be a shortest path with the property that its first two vertices are the same as the last two but in reversed order, if such a path exists. Write $P_0 := v_0 \ldots v_k$ and denote $x := v_0 = v_k$ and $y := v_1 = v_{k-1}$.
		
		Define $\sigma \in \{0, 1, 2\}$ by
		\begin{equation} \label{eq:sigma-def}
			\sigma \modthree{2k}.
		\end{equation}
		
		Intuitively, $\sigma$ is defined so that if $P_0$ has a good colouring, then the apexes in this colouring are at index $\sigma \pmod{3}$.
		If $P_0$ does not exist, define $\sigma = 2$.
		
		Let $\{z,w\}$ be an edge in the shadow of $\HH$, and let $P = xy v_2 v_3 \ldots v_k$ be a path from $xy$ whose last three vertices contain $z$ and $w$. Let $i_w \in \{k-2,k-1, k\}$ be the index of $w$ (namely, $v_{i_w} = w$), and define $i_z$ analogously, and note that $i_w \neq i_z$. Define the index
		$$\eta(P, \{z,w\}) = \begin{cases}
			w, & i_w \modthree{\sigma}, \\
			z,  &i_z \modthree{\sigma},\\
			*, &\text{otherwise}.
		\end{cases}$$
		
		In particular, this defines $\eta(xy, \{x, y\})$.  
		We claim that $\eta(P, \{z,w\})$ is independent of the choice of the path $P$.
		
		\begin{claim} \label{claim:consistent}
			Let $z, w \in V(\HH)$ be distinct.
			Let $P = v_0 \ldots v_p$ and $Q = u_0 \ldots u_q$ be two paths 
			starting at $xy$ such that $z$ and $w$ are among their last three vertices. Then
			$$\eta(P, \{z, w\}) = \eta(Q, \{z, w\}).$$
		\end{claim}
		\begin{proof}
			Let $i_z \in \{p-2, p-1, p\}$ such that $v_{i_z} = z$, and define $j_z$ similarly with respect to $Q$.
			It suffices to prove that $i_z \modthree{\sigma}$ if and only if $j_z \modthree{\sigma}$. Indeed, this implies the same equivalence for $w$, thus showing that $\eta(P, \{z, w\})  = *$ if and only if $ \eta(Q, \{z, w\})= *$.
			
			First, we modify $P$ so as to assume that $P$ ends with $zw$ or $wz$. If this is not the case, then up to swapping $z$ and $w$ we have that $P$ ends with either $zw*$ or $z*w$. In the former case remove the last vertex of $P$, and in the latter case append $z$ to $P$. It is easy to see that the statement of the claim holds for the original $P$ if and only if it holds for the modified path. Similarly, we may assume that $Q$ ends with $zw$ or $wz$.
			
			Assume first that $P$ and $Q$ both end with $zw$. Then $\tilde{Q}$ (defined as in \eqref{eqn:tilde}) is a path from $zw$ to $xy$ of order $2(q+1)-2 = 2q$. Hence $v_2 v_3 \ldots v_{p-2} \tilde{Q}$ is a cycle, and by assumption its order is divisible by 3. That is, $p-3+2q \modthree{0}$, and thus $p \modthree{q}$. Since $i_z = p-1$ and $j_z = q-1$, this proves \Cref{claim:consistent}. The same argument holds when $P$ and $Q$ both end with $wz$.
			
			Secondly, assume that $P$ is a path from $xy$ to $zw$ and $Q$ is a path from $xy$ to $wz$. Note that this case only arises if $P_0$ is defined, as $v_0 v_1 \ldots v_{p-2} zw u_{q-2}\dots u_0$ is a path from $xy$ to $yx$. Then consider the cycle $v_2 \ldots v_{p-2} zw u_{q-2} \ldots u_2 \tilde{P_0}$. This is indeed a cycle because $u_1 u_0 = yx$, $v_0 v_1 = xy$ and $\tilde{P_0}$ is a path from $yx$ to $xy$. The order of this cycle is $p-3+2+q-3+2k \modthree{p+q+2+\sigma}$, using~\eqref{eq:sigma-def}. Now substitute $p = i_z+1$ and $q = j_z$. We have $i_z + j_z +\sigma \modthree{0}$, so $i_z \modthree{\sigma}$ if and only if $j_z \modthree{\sigma}$.
		\end{proof}
		
		Note that for every edge $\{z,w\}$ in the shadow of $\HH$ there is a path $P$ starting at $xy$ whose last three vertices contain $z$ and $w$. Indeed, as $\HH$ is tightly connected, there is a path $Q$ such that $x$ and $y$ are among its first three vertices and $z$ and $w$ among its last three vertices. Using a modification as in the proof of \Cref{claim:consistent} we may assume that $Q$ starts with $xy$ or $yx$. If it starts with $xy$ we are done, and otherwise the reverse of the path $\tQ$ satisfies the requirements.
		
		Given an edge $\{z,w\}$ in the shadow of $\HH$, define $\eta(zw) = \eta(P, \{zw\})$, where $P$ is any path from $xy$ whose last three vertices contain $z$ and $w$ (which exists by the previous paragraph). This parameter is well defined by \Cref{claim:consistent}. Now define $\chi$ as follows: let $zw$ be blue if $\eta(zw) = *$, and let it be red and oriented away from $\eta(zw)$ otherwise.
		
		Finally, we show that $\chi$ is a good colouring. To see this, consider an edge $uvw$ of $G$. 
		Let $P$ be a path from $xy$ whose last three vertices are $u, v, w$ (in some order); such a path exists by the paragraph above. Write $P:=xy v_2 v_3 \ldots v_{p-2}v_{p-1}v_p$, let $i \in \{p-2, p-1, p \}$ with $i \modthree{\sigma}$, and we may assume that $v_i =u$. Then $\eta(uv)=\eta(uw) =u$ and $\eta(vw) = *$, which implies that $uvw$ is a cherry with apex $u$.	
	\end{proof}
	
	\begin{remark}
		Our proof actually shows that if $G$ does not contain a path $P_0$ starting and ending with $xy$ and $yx$ respectively, then the graph is tripartite.

		Notice that a good colouring of $\HH$ can be extended from the shadow of $\HH$ to $K_n$ with no restrictions. Thus, in what follows it will suffice to analyse colourings of complete graphs by blue edges and red oriented edges (we will call such graphs \emph{coloured graphs}).
	\end{remark}
	
	\section{Maximising the number of cherries} \label{sec:optimal}
	The results of the previous section establish a connection between maximising the number of edges in an odd-pseudocycle-free hypergraph and maximising the number of \textit{cherries} in colourings of $K_n$ (formally defined below). It will turn out that both problems have the same extremal construction which yields the maximum $f(n)$. Recall that we have defined $f(n)$ as the maximum number of edges in a hypergraph $\HH(x_1, \ldots, x_k)$ with $\sum_i{x_i} = n$. An explicit expression for $f$ is
	\begin{equation} \label{eqn:f}
		f(n) = \max_{k \ge 1} \max_{\substack{x_1, \ldots, x_k \ge 1 : \\  x_1 + \ldots + x_k = n}} \left\{\sum_{1 \le i < j \le k} \binom{x_i}{2} \cdot x_j\right\}.
	\end{equation}
	Equivalently, we have the recursive characterisation
	\begin{align} \label{eqn:k}
		\begin{split}
			& f(1) = 0,\\
			& f(n) = \max_{k \in [n-1]} \binom{k}{2} (n-k) + f(n-k) \,\,\text{ for $n \ge 2$}.
		\end{split}
	\end{align}
	Write 
	\begin{equation} \label{eq:alphabeta}
		\beta = \frac{3 - \sqrt{3}}{2} \approx 0.634 \qquad \text{and} \qquad
		\alpha = \frac{\beta(1 - \beta)}{2(3 - 3\beta + 3\beta^2)} = \frac{\sqrt{3}}{3} - \frac{1}{2} \approx 0.077.
	\end{equation}
	
	The following proposition will be proved in \Cref{subsec:calculus}.
	\begin{proposition} \label{prop:alpha}
		$f(n) = \alpha n^3 + o(n^3)$.
	\end{proposition}
	
	We remark that the density of the corresponding hypergraph $\HH_n$ is $6\alpha = 2\sqrt{3}-3$, as already noted by  Mubayi and R\"odl~\cite{mubayi2002turan}.
	
	As in Section~\ref{sec:overview}, we call a graph $G$ \emph{coloured} if it is a complete graph whose edges are either coloured blue or coloured red and oriented.
	A \emph{cherry} in a coloured graph $G$ is a triple $xyz$ such that $xy$ and $xz$ are red and directed from $x$ and $yz$ is blue. Denote by $c(G)$ the number of cherries in $G$. Theorem~\ref{theorem_intro_falgas_ravry} states that $c(G) \leq f(n)$ for any $n$-vertex coloured graph $G$. Recall that this was originally proved by Falgas-Ravry and Vaughan \cite{falgas2012turan} (who used flag algebras and also proved a similar result for out-directed stars on four vertices) and by Huang \cite{huang2014maximum} (who used a symmetrisation argument, and proved a similar result for out-directed stars on $k$ vertices, for every $k \ge 3$). Nevertheless, we provide a proof, both for completeness and because we need most of the groundwork to prove a stability version of Theorem~\ref{theorem_intro_falgas_ravry}.
	
	As mentioned in the proof overview (\Cref{sec:overview}), \Cref{corollary_intro_long_pseudocycles}, which is a weak version of our main result and is restated here, follows directly from \Cref{theorem_intro_falgas_ravry} (proved in the next section) and \Cref{thm:good-colouring} (proved in the previous section).
	
	\corLongPseudocylces*

	\section{Stability with symmetrisation} \label{sec:symm}
	
	Most of the work in this section will go into proving the following lemma, providing a stability version of Theorem~\ref{theorem_intro_falgas_ravry}.
	It will then be iterated to prove a stability result about cherries in coloured graphs; recall that $\beta = (3 - \sqrt{3})/2$ (see \eqref{eq:alphabeta}).
	
	We point out that this stability result is somewhat similar to a general result due to Liu--Pikhurko--Sharifzadeh--Staden \cite{liu2020stability} which allows one to obtain stability versions of a class of extremal results that can be proved using a symmetrisation argument. However, while we indeed prove the extremal result in Theorem~\ref{theorem_intro_falgas_ravry} using a symmetrisation argument, the result in \cite{liu2020stability} does not apply to automatically convert it into a stability result.
	
	\begin{lemma} \label{lem:ind-step}
		Let $ 1/n \ll \eps \ll 1$ and let $G$ be a coloured graph on $n$ vertices satisfying $c(G) \ge f(n) - \eps^2 n^3$.
		Then there is a coloured graph $G'$ on $V(G)$ satisfying: $c(G') \ge c(G)$; the graphs $G$ and $G'$ differ on at most $800\eps^{1/2} n^2$ edges; moreover, there is a set $Q \subseteq V(G)$ satisfying $\big| |Q| - \beta n \big| \le 100\eps n$; $Q$ is a blue clique in $G'$; and all other edges in $G'$ that are incident with $Q$ are red and oriented towards $Q$. 
	\end{lemma}
	
	The proof consists of two main parts: first we show that $G$ has a blue almost-clique on a vertex set $Q'$ of size roughly $\beta n$. Then we show that most $(V \setminus Q', Q')$ edges are red and point towards $Q'$. In both parts, we make use of a ``symmetrisation procedure'' which builds blue cliques without decreasing the number of cherries.
	
	A \emph{blue clone-clique} in a coloured graph $G$ is a set of vertices $Q$ such that $Q$ is a blue clique in $G$, and for any $v \notin Q$, either all edges between $v$ and $Q$ are blue, or they are all red and have the same orientation (namely, they all point towards $v$ or all point away from $v$). A \emph{full blue clone-clique} is a blue clone-clique $Q$ such that all $(V \setminus Q, Q)$ edges are red.
	
	The symmetrisation procedure, which will be described in detail in the next section, receives as input a vertex $x$ in a graph $G$, and produces a graph $G'$ on the same vertex set, which has at least as many cherries as $G$ and has a full blue clone-clique $Q$ in $G'$ that contains $x$.
	
	The symmetrisation procedure can be applied repeatedly to a coloured graph $G$ to find a coloured graph $G'$ with at least as many cherries as $G$, and whose vertices can be partitioned into full blue clone-cliques. Some calculus (detailed in \Cref{subsec:calculus}) will show that such a $G'$ contains a full blue clone-clique $Q'$ of size approximately $\beta n$. 
	
	To proceed we need two lemmas (\Cref{lem:Q-red,lem:Q-blue}; see \Cref{subsec:Q-lemmas}) that together tell us the following. Suppose that a symmetrisation procedure on $G$ resulted in a full blue clone-clique $Q$, of size approximately $\beta n$. Then (even before performing symmetrisation) almost all edges in $G[Q, V \setminus Q]$ are red and point towards $Q$,  and almost all edges in $G[Q]$ are blue.
	
	Applying these lemmas to the previously found blue clone-clique $Q'$, we conclude first that $G[Q']$ is almost fully blue. 
	We then show that there is a particular instance of the symmetrisation procedure that results in a graph $G'$ and full blue clone-clique $Q$ such that $Q$ and $Q'$ differ on only few vertices. Lemma~\ref{lem:Q-red} implies that almost all $G[Q', V \setminus Q']$ edges are red and point towards $Q'$. This essentially completes the proof. This part is detailed in \Cref{subsec:proof}.
	
	In \Cref{subsec:iteration}, we iterate \Cref{lem:ind-step} to prove the following result.
	\begin{theorem}\label{thm:stability}
		Let $1/n \ll \eps_1 \ll \eps_2 \ll 1$.
		Let $G$ be a coloured graph on $n$ vertices satisfying $c(G) \ge f(n) - \eps_1 n^3$.
		Then there exists a coloured graph $G'$ on the same vertex set, satisfying: 
		\begin{enumerate}[label = \rm(\alph*), ref = \rm (\alph*)]
			\item \label{itm:stability-a}
			$c(G') \ge c(G)$,
			\item \label{itm:stability-b}
			$G$ and $G'$ differ on at most $\eps_2 n^2$ edges,
			\item \label{itm:stability-c}
			the vertices of $G'$ can be partitioned into $Q_1, \ldots, Q_t$ such that:
			\begin{enumerate}[label = \rm(\roman*), ref = c\rm(\roman*)]
				\item \label{itm:stability-c1}
				$|Q_1| \ge \ldots \ge |Q_t|$, 
				\item \label{itm:stability-c2}
				all edges in $Q_i$ are blue, for $i \in [t]$, 
				\item \label{itm:stability-c3}
				all edges in $(Q_i, Q_j)$ are red and directed towards $Q_i$, for $1 \le i < j \le t$,
				\item \label{itm:stability-c4} 	$\big||Q_i| - \beta \cdot |Q_i \cup \ldots \cup Q_t| \big| \le \eps_2 n$ for $i \in [t]$.
			\end{enumerate}
		\end{enumerate}
	\end{theorem}
	
	\def \Nm {N^-}
	\def \Np {N^+}
	In a coloured graph $G$, let $\Nm_G(x)$ be the red in-neighbourhood of $x$ and let $\Np_G(x)$ be the red out-neighbourhood of $x$ (we sometimes omit the subscript $G$).

	\subsection{The symmetrisation procedure} \label{subsec:symmproc}
	Given $x \in V(G)$, the \emph{symmetrisation procedure} $S_G(x)$ (or $S(x)$ in short) builds a blue clone-clique containing $x$; see \Cref{fig:symmetrisation} for a detailed description. The result of the procedure depends on the choice of $x_{k+1}$ in step \ref{step:4}, but we suppress this dependence in the notation $S_G(x)$.
	
	\begin{figure}[h!t!]
		\begin{framed}
			\begin{minipage}{.98\textwidth}
				\begin{center}
					{\bf The symmetrisation procedure $S_G(x)$}
				\end{center}
				
				{\bf Input:} a coloured graph $G$ on vertex set $V$ and a vertex $x \in V$.
				
				\vspace{.2cm}
				{\bf Output:} a graph $G'$ and a full blue clone-clique $Q$ in $G'$ containing $x$.
				
				\vspace{.2cm}
				{\bf The process:} the algorithm builds sequences $x_1, \ldots, x_t$ and $y_1, \ldots, y_t$ of vertices in $V$ and $G_1, \ldots, G_t$ of graphs on $V$, for some $t$, as follows.
				\begin{enumerate}
					\item 
					Set $x_1, y_1 := x$ and $G_1 := G$.
					\item \label{step:2}
					Suppose $x_1, \ldots, x_k$ and $G_1, \ldots, G_k$ are given and $\{x_1, \ldots, x_k\}$ is a blue clone-clique in $G_k$. 
					\item \label{step:3}
					If there are no vertices in $V \setminus \{x_1, \ldots, x_k\}$ whose edges to $\{x_1, \ldots, x_k\}$ are all blue, put $Q := \{x_1, \ldots, x_k\}$ and $G' := G_k$, and return $G'$ and $Q$.
					\item \label{step:4}
					Otherwise, let $x_{k+1}$ be a vertex in $V \setminus \{x_1, \ldots, x_k\}$ which sends blue edges to $\{x_1, \ldots, x_k\}$ ($x_{k+1}$ can be chosen arbitrarily or judiciously).
					\item
					For $y \in V$, let $G_{k+1}(y)$ be the graph obtained by replacing $\{x_1, \ldots, x_{k+1}\}$ by $k+1$ copies of $y$ and letting the new vertices form a blue clique.
					
					If $c(G_{k+1}(x_{k+1})) - c(G_k) \ge k \cdot (c(G_{k+1}(x_1)) - c(G_k))$, set $y_{k+1} = x_{k+1}$, and otherwise let $y_{k+1} = x_1$.
					Define $G_{k+1} := G_{k+1}(y_{k+1})$, and return to step \ref{step:2}, considering the sequences $x_1, \ldots, x_{k+1}$ and $G_1, \ldots, G_{k+1}$.
				\end{enumerate}
			\end{minipage}
		\end{framed}\vspace{-0.5cm}
		\caption{Description of the symmetrisation process $S_G(x)$}
		\label{fig:symmetrisation}
	\end{figure}
	
	We now show that the procedure $S_G(x)$ does not decrease the number of cherries. In fact, we prove a stronger quantitative claim.
	
	\begin{claim}  \label{claim:nhoods}
		
		Let $x_1, \ldots, x_t$, $y_1, \ldots, y_t$ and $G_1, \ldots, G_t$ be sequences produced by $S_G(x)$, let $k \in [t-1]$, and use $\Nm(u)$ as a shorthand for $\Nm_{G_k}(u)$. Then, one of the following holds.
		\begin{enumerate}[label = \rm(\roman*)]
			\item \label{itm:nhoods1}
			$y_{k+1} = x_1$ and $c(G_{k+1})-c(G_k) \geq \frac{k+1}{4} \cdot \big| \Nm(x_1) \,\triangle\, \Nm(x_{k+1})\big|$,
			\item \label{itm:nhoodsk}
			$y_{k+1} = x_{k+1}$ and $c(G_{k+1})-c(G_k) \geq \frac{k(k+1)}{4} \cdot \big| \Nm(x_1) \, \triangle \, \Nm(x_{k+1})\big|$.
		\end{enumerate}
		In particular, $c(G_{k+1}) \ge c(G_k)$.
	\end{claim}
	
	\begin{proof}
		Let $c(x_i)$ denote the number of cherries in $G_k$ containing $x_i$ and no other vertices in $x_1, \ldots, x_k$. 
		Recall that for $y \in \{x_1, \ldots, x_{k+1}\}$ the graph $G_{k+1}(y)$ is obtained from $G_k$ by replacing $x_1, \ldots, x_{k+1}$ by copies of $y$ that form a blue clique. Write $\Delta_j := c(G_{k+1}(x_j)) - c(G_k)$.
		We claim that
		\begin{equation*}
			\Delta_j = 
			(k+1)c(x_j) - \sum_{i \in [k+1]} c(x_i) + \binom{k+1}{2}\big|N^-(x_j)\big| - \frac 12  \sum_{i_1 \neq i_2} \big|N^-(x_{i_1}) \cap N^-(x_{i_2})\big|.
		\end{equation*}
        Indeed, the first two terms account for the triples with only one vertex in $\{x_1, \ldots, x_k\}$. For the second two terms, notice that a triple $(x_{i_1}, x_{i_2}, v)$ with $1 \leq i_1 < i_2 \leq k+1$ is a cherry in $G_{k+1}(x_j)$ if and only if  $v  \in N^-(x_j) $, and it is a cherry in $G_k$ if and only if $v \in N^-(x_{i_1}) \cap N^-(x_{i_2})$.
		Summing over $j \in [k+1]$, we obtain
		\begin{align*}
			\sum_j \Delta_j 
			& = \binom{k+1}{2}\sum_j \big|N^-(x_j)\big| - \frac{k+1}{2}\sum_{i \neq j}\big|N^-(x_i) \cap N^-(x_j)\big| \\
			& = \frac{k+1}{2}\sum_{i \neq j} \big|N^-(x_j) \setminus N^-(x_i)\big|.
		\end{align*}
		
		In particular, since $\{x_1, \dots, x_k\}$ is a blue clone-clique,
		\begin{align*}
			k \Delta_1 + \Delta_{k+1} 
			& = \frac{k(k+1)}{2} \cdot \left( \big|N^-(x_{k+1}) \setminus N^-(x_1)\big| + \big|N^-(x_{1}) \setminus N^-(x_{k+1})\big|\right) \\
			& = \frac{k(k+1)}{2} \cdot \big|\Nm(x_1) \,\triangle\, \Nm(x_{k+1})\big|.
		\end{align*}
		
		Now, $\max(k \Delta_1, \Delta_{k+1})$ is at least one half of the RHS. Thus, if $y_{k+1} = x_{k+1}$ then $\Delta_{k+1} \ge k \Delta_1$ and so $\Delta_{k+1} = \max(k \Delta_1, \Delta_{k+1}) \ge \frac{k(k+1)}{4} \cdot \big|\Nm(x_1) \,\triangle\, \Nm(x_{k+1})\big|$, and if $y_1 = x_1$ then $k \Delta_1 > \Delta_{k+1}$ and so 
		$\Delta_1 = \max(k\Delta_1, \Delta_{k+1}) \geq \frac{k+1}{4} \cdot |\Nm(x_1) \,\triangle\, \Nm(x_{k+1})\big|$.
	\end{proof}
	
	\Cref{theorem_intro_falgas_ravry} follows easily from the above claim.
	\def \Gf {G_{\final}}
	\begin{proof}[Proof of Theorem~\ref{theorem_intro_falgas_ravry}]
		Let $G$ be a coloured graph on $n$ vertices. Run the following process: starting with $G' = G$, as long as there is a vertex $x$ which is not in a full blue clone-clique in $G'$, run $S_{G'}(x)$ and replace $G'$ by the resulting graph. Let $\Gf$ be the graph $G'$ at the end of the process (notice that the process will indeed end, because $S_{G'}(x)$ keeps full blue clone-cliques intact). Then $c(\Gf) \ge c(G)$ by \Cref{claim:nhoods}, and the vertices of $\Gf$ can be partitioned into full blue clone-cliques $Q_1, \ldots, Q_t$; for convenience suppose that $|Q_1| \ge \ldots \ge |Q_t|$. Replace $\Gf$ by the graph $\Gf'$ obtained by directing the red edges between $Q_i$ and $Q_j$ towards $Q_i$, for $1 \le i < j \le t$. It is straightforward to verify that $c(\Gf') \ge c(\Gf)$, as the number of cherries in $Q_i \cup Q_j$ is larger when the arcs in $(Q_i, Q_j)$ point towards the larger clique. Finally, denoting $q_i := |Q_i|$, observe that $c(\Gf') = \sum_{i \le j} \binom{q_i}{2}q_j \le f(n)$ (see \eqref{eqn:f}). Thus $c(G) \le f(n)$, as claimed.
	\end{proof}

	\subsection{Optimising the clique size} \label{subsec:calculus}
	
	Before proceeding to analyse the symmetrisation procedure, we prove the following lemma regarding the structure of a graph whose vertices are partitioned into full blue clone-cliques, mostly using calculus; recall that $\alpha$ and $\beta$ are defined in \eqref{eq:alphabeta}.
	
	\begin{lemma} \label{lem:theta}
		Let $1/n \ll \eps \ll 1$.
		Let $G$ be a coloured graph on $n$ vertices whose vertices can be partitioned into full blue clone-cliques, and suppose that $c(G) \ge f(n) - \eps^2 n^3$. Then $G$ has a full blue clone-clique $Q$ satisfying $\big| |Q| - \beta n \big| \le 100 \eps n$.
	\end{lemma}
	
	Define a function $g:[0,1] \to \mathbb{R}$ as follows.
	\begin{equation} \label{eqn:g}
		g(x) = \frac{x(1-x)}{2(3-3x+x^2)}.
	\end{equation}
	It will be convenient to note the following equation.
	\begin{equation} \label{eqn:g-var}
		\left(1 - (1 - x)^3\right) \cdot g(x) = \frac{1}{2} \cdot x^2(1 - x).
	\end{equation}
	One can check that $g'$ is decreasing and $g'(\beta) = 0$, showing that
	\begin{equation} \label{eqn:g-beta}
		g(x) \le g(\beta) = \alpha \qquad \text{for $x \in [0,1]$}.
	\end{equation}
	
	We first prove \Cref{prop:alpha} regarding the value of $f(n)$.
	\begin{proof}[Proof of \Cref{prop:alpha}]
		We show that $f(n) \leq \alpha n^3$ by induction on $n$.  This is true for $n=1$. Suppose that $f(m)\leq  \alpha m^3$ for $m< n$.  
		
		Given $k \in [n-1]$ that maximises the LHS in \eqref{eqn:k}, write $x = k/n$. The recursive definition of $f$ implies that  $ \frac{f(n)}{n^3} \leq \frac 12 \cdot x^2(1-x) + \alpha(1-x)^3$. Subtracting $\alpha$ and using \eqref{eqn:g-var}, we obtain
		\begin{equation*}
			\frac{f(n)}{n^3}- \alpha \leq \frac 12 \cdot x^2(1-x) - \alpha(1-(1-x)^3) = (1-(1-x)^3)(g(x)-\alpha) \le 0,
		\end{equation*}
		as required.
		
		To verify that $f(n) \geq (\alpha + o(1))n^3$, set $x_i =\lfloor \beta (1-\beta)^i n \rfloor$ in~\eqref{eqn:f}.
	\end{proof}
	
	\begin{proof}[Proof of \Cref{lem:theta}]
		Let $Q_1, \ldots, Q_t$ be the full blue clone-cliques in $G$, arranged in descending order according to their sizes. Let $G'$ be obtained from $G$ by orienting the $(Q_i, Q_j)$ (red) edges towards $Q_i$, for $1 \le i < j \le t$. As explained before and by assumption on $G$, $c(G') \ge c(G) \ge f(n) - \eps^2 n^3$.
		
		Notice that $|Q_1| \ge 0.01 n$, because otherwise $c(G) \le n^2|Q_1| \le 0.01 n^3 < f(n) - \eps^2 n^3$ (recall that $f(n) \approx 0.077n^3$, by \Cref{prop:alpha}).
		
		Write $|Q_1| = \theta n$. Then, using $f(n) = \alpha n^3 + o(n^3) = g(\beta) n^3 + o(n^3)$ (which follows from \Cref{prop:alpha} and the definition of $\alpha$ in \eqref{eq:alphabeta}),
		\begin{equation*}
			c(G') \le \binom{|Q_1|}{2}(n - |Q_1|) + f(n - |Q_1|)
			\le \frac{1}{2} \theta^2 (1 - \theta)n^3 + g(\beta)(1 - \theta)^3n^3 + o(n^3).
		\end{equation*}
		Thus, using \eqref{eqn:g-var},
		\begin{align*}
			\eps^2 \ge \frac{f(n) - c(G')}{n^3} 
			& \ge g(\beta) - g(\beta)(1 - \theta)^3 - \frac{1}{2}\theta^2(1 - \theta) + o(1) \\
			& = \theta \cdot (3 - 3\theta + \theta^2) \cdot (g(\beta) - g(\theta)) + o(1) \\
			& \ge 0.02 \cdot (g(\beta) - g(\theta)) + o(1).
		\end{align*}
		For the last inequality we used $\theta \ge 0.01$, which implies $\theta(3-3\theta+\theta^2) \ge 0.02$. 
		By bounding the $o(1)$ term by $\eps^2/2$ and using \Cref{claim:calculus} below, we get
		\begin{equation*}
			100\eps^2 \ge g(\beta) - g(\theta) \ge \min\{0.05(\beta - \theta)^2, 0.005\}.
		\end{equation*}
		Since $\eps$ is very small, we get $100 \eps^2 \ge 0.05(\beta - \theta)^2$, which implies $|\beta - \theta| \le 100\eps$.
	\end{proof}
	
	\begin{claim} \label{claim:calculus}
		For $x \in [0,1]$,
		\begin{equation}
			g(\beta) - g(x) \geq \min\{0.05(\beta-x)^2, 0.005\}.
		\end{equation}
	\end{claim}  
	\begin{proof}
		We use the following facts that can be checked easily.
		\begin{itemize}
			\item
			The function $g(x)$ is increasing on $[0,\beta]$ and decreasing on $[\beta, 1]$. In particular, its maximum is attained at $\beta$, and $g'(\beta) = 0$.
			\item
			$g(\beta) - g(0.5) \ge 0.005$.
			\item
			The second derivative $g''(x)$ (which is $\frac{-x(2x^2 - 9x + 9)}{(x^2 - 3x + 3)^3}$) is non-negative and decreasing on $[0,1]$. In particular $g''(x) \le g''(0.5) \le -0.4$ for $x \in [0.5, 1]$.
			\item
			By Taylor's expansion: $g(x) = g(\beta) + \frac{1}{2}g'(\beta)(x - \beta) + \frac{1}{6}g''(c_x)(x - \beta)^2$ for every $x \in [0,1]$ and some $c_x$ between $x$ and $\beta$.
		\end{itemize}
		
		By the first and second items (using $\beta > 0.5$), if $x \in [0, 0.5]$ then 
		\begin{equation*}
			g(\beta) - g(x) \ge g(\beta) - g(0.5) \ge 0.005.
		\end{equation*}
		By the first, third and fourth items, if $x \in [0.5, 1]$, then 
		\begin{equation*}
			g(\beta) - g(x) \ge \frac{0.4}{6}(x - \beta)^2 \ge 0.05(x - \beta)^2.
		\end{equation*}
		The two inequalities prove the claim.
	\end{proof}

	\subsection{Blue clone-cliques before and after symmetrisation} \label{subsec:Q-lemmas}
	The next two lemmas show that if a symmetrisation procedure on $G$ produces a full blue clone-clique $Q$ of size approximately $\beta n$, then almost all edges in $G[Q, V \setminus Q]$ are red and oriented towards $Q$ and almost all edges in $G[Q]$ are blue.
	
	\begin{lemma} \label{lem:Q-red}
		Let $1/n \ll \eps \ll 1$.
		of a procedure $S_G(x)$, and suppose that $|Q| \ge 0.55n$.
		Then all but at most $10\eps n^2$ edges in $G[Q, V \setminus Q]$ are red and directed towards $Q$.
	\end{lemma}
	\begin{proof}
		\def \Vi{V_{\inn}} 
		\def \Vo{V_{\out}} 
		Set $U := V \setminus Q$, let $\Vi$ be the set of vertices $u$ in $U$ for which $uq$ is a red arc in $G'$ for every $q \in Q$, and let $\Vo := U \setminus \Vi$.
		We will show that $\Vo$ is small, and that not many pairs incident to $\Vi$ were recoloured during the symmetrisation procedure $S_G(x)$.
		
		First, we show $|\Vo| \leq 40\eps^2 n$.
		Let $G''$ be obtained from $G'$ by reorienting the edges in $G'[Q, \Vo]$ to point towards $Q$. Notice that the cherries in $G'$ that contain an edge in $(Q, \Vo)$ consist of one vertex in $Q$ and two in $\Vo$, and thus their number is at most $|Q|\binom{|\Vo|}{2}$. Also, every set consisting of two vertices in $Q$ and one in $\Vo$ is a cherry in $G''$ but not in $G'$. Thus, using $|Q| \ge 0.55 n$ which implies $|Q| - |\Vo| \ge 0.1n$,
		\begin{align*}
			c(G'') - c(G') 
			& \ge \binom{|Q|}{2}|\Vo| - \binom{|\Vo|}{2}|Q| 
			= \frac 12 |Q||\Vo|(|Q| - |\Vo|) \\
			& \ge \frac{1}{2} \cdot \frac{n}{2} \cdot \frac{n}{10} \cdot |\Vo|
			= \frac{n^2}{40} \cdot |\Vo|.
		\end{align*}
		Recall that $c(G) \ge f(n) - \eps^2 n^3$ by assumption, $c(G') \ge c(G)$ by \Cref{claim:nhoods}, and $c(G'') \le f(n)$ by Theorem~\ref{theorem_intro_falgas_ravry}. Altogether, this implies $c(G'') - c(G') \le \eps^2 n^3$ and thus $|\Vo| \le 40\eps^2 n$, as claimed.
		
		Let $R$ be the set of edges $qv$ in $(Q, V \setminus Q)$ that are red and oriented towards $Q$ in $G'$ but not in $G$. 
		We now upper-bound $|R|$. Notice that each such edge in $R$ was recoloured to a red arc oriented towards $Q$ at some point during $S_G(x)$ (possibly more than once).
		Let $G = G_1, \ldots, G_{t} = G'$ be the graphs obtained during the symmetrisation process on $Q$ and let $x_1, \ldots, x_t$ be the corresponding sequence of vertices.
		For each $v \in V$ and $k \in [t]$, let $A_k(v)$ be the set of ordered pairs $vq$ which changed to red arcs in step $k$ (so they were recoloured from $G_{k-1}$ to $G_k$).
		
		We claim that $\sum_{k \geq \eps n} \sum_{v \in \Vi} |A_k(v)| \leq 4 \eps n^2$.
		To see this, fix $k \geq \eps n$ and consider the $k$-th step. If $y_k = x_1$, then $A_k(v) = \{vx_k\}$ for $v \in \Nm(x_1) \setminus \Nm(x_k)$ and $A_k(v) = \emptyset$ otherwise, where $\Nm(\cdot)$ refers to the in-neighbourhood with respect to $G_{k-1}$. Thus, using \Cref{claim:nhoods}~\ref{itm:nhoods1},
		\begin{equation*}
			\sum_{v \in \Vi} |A_k(v)| 
			\leq \big|N^-(x_1) \setminus N^-(x_k)\big| 
			\leq \frac{4}{k} \cdot \big(c(G_k) - c(G_{k-1})\big).
		\end{equation*}
		If $y_k = x_k$, then $A_k(v) = \{vx_1, \ldots, vx_{k-1}\}$ for $v \in \Nm(x_k) \setminus \Nm(x_1)$ and $A_k(v) = \emptyset$ otherwise. Thus, by \Cref{claim:nhoods}~\ref{itm:nhoodsk},
		\begin{equation*}
			\sum_{v \in \Vi} |A_k(v)| 
			\leq (k-1) \cdot \big|N^-(x_1) \setminus N^-(x_k)\big| 
			\leq \frac{4}{k} \cdot \big(c(G_k) - c(G_{k-1})\big),
		\end{equation*}
		In either case, we get that for  $k \geq \eps n$, 
		\begin{equation*}
			\sum_{v \in \Vi} |A_k(v)| \leq \frac{4}{\eps n}\big(c(G_k) - c(G_{k-1})\big).
		\end{equation*}
		Summing over $k \geq \eps n$, we obtain the required inequality 
		\begin{equation*}
			\sum_{k \ge \eps n} \sum_{v \in \Vi} |A_k(v)| \leq \frac{4}{\eps n}\big(c(G') - c(G_{\eps n})\big) \leq 4 \eps n^2,
		\end{equation*}
		Where the last equality holds since $c(G') - c(G_{\eps n}) \leq \eps^2 n^2$.
		
		Note that $|R| \le \eps n^2 + \sum_{k \ge \eps n} \sum_{v \in \Vi} |A_k(v)| \le 5\eps n^2$.
		In total, all but at most $(40 \eps^2 + 5\eps)n^2 \leq 10 \eps n^2$ pairs in $(Q, V \setminus Q)$ are red and oriented towards $Q$.
	\end{proof}

	\begin{lemma} \label{lem:Q-blue}
		Let $1/n \ll \eps \ll 1$.
		Let $G$ be a coloured graph on $n$ vertices with at least $f(n) - \eps^2 n^3$ cherries. Suppose that $G'$ and $Q$ are the graph and full blue clone-clique produced by the procedure $S_G(x)$, and suppose that $0.55n \le |Q| \le 0.65n$.
		Then all but $1200\eps n^2$ edges in $G[Q]$ are blue.
	\end{lemma}
	
	\begin{proof}
		Let $F$ and $F'$ be the graphs obtained from $G$ and $G'$ by colouring all $(Q, V \setminus Q)$ edges red and orienting them towards $Q$. Notice that $F'$ can be obtained from $F$ by colouring all edges in $Q$ blue.
		
		We will first derive an upper bound on $c(F') - c(F)$. By \Cref{lem:Q-red}, the graphs $G$ and $F$ differ on at most $10\eps n^2$ edges and thus $|c(G) - c(F)| \le 10\eps n^3$. Similarly, $|c(G') - c(F')| \le 10\eps n^3$ (the lemma is still applicable, as $S_{G'}(x)$ does not change the graph $G'$). By assumption on $G$ we also have $c(G') - c(G) \le \eps^2 n^3$. Altogether,
		\begin{equation} \label{eqn:upper}
			c(F') - c(F) \le c(G') - c(G) + 20\eps n^3 \le (\eps^2 + 20\eps)n^3 \le 30 \eps n^3.
		\end{equation}
		
		\def \dpp {d^+}
		We now obtain a lower bound on the same quantity.
		Let $e$ be the number of red edges in $G[Q]$.
		The number of cherries in $F$ that are not cherries in $F'$ is at most $\sum_{q \in Q} \binom{\dpp(q)}{2}$, where $\dpp(q)$ denotes the red out-degree of $q$ in $F[Q]$. Notice that 
		\begin{equation*}
			\sum_{q \in Q} \binom{\dpp(q)}{2} \le \frac{1}{2} \sum_{q \in Q}(\dpp(q))^2 \le \frac 12 e|Q|, 
		\end{equation*}
		because $\dpp(q) \le |Q|$ and $e = \sum_q \dpp(q)$.
		
		On the other hand, the number of cherries in $F'$ that are not cherries in $F$ is exactly $e(n - |Q|)$. Thus,
		\begin{equation} \label{eqn:lower}
			c(F') - c(F) 
			\ge e(n - |Q|) - \frac 12 e|Q| 
			= e\cdot \big(n - \frac{3}{2}|Q|\big) 
			\ge \frac{en}{40},
		\end{equation}
		using $|Q| \le 0.65 n$.
		
		By \eqref{eqn:upper} and \eqref{eqn:lower}, we have $e \le 1200 \eps n^2$, as claimed.
	\end{proof}

	\subsection{Proof of \Cref{lem:ind-step}} \label{subsec:proof}
	
	Finally, we start with the actual proof of \Cref{lem:ind-step}. The first step is to find a set $Q'$ of the right size almost all of whose edges in $G$ are blue.
	
	\begin{lemma} \label{lem:beta-clique}
		Let $1/n \ll \eps \ll 1$.
		Let $G$ be a coloured graph on $n$ vertices, satisfying $c(G) \ge f(n) - \eps^2 n^3$.
		Then there is a set $Q' \subseteq V(G)$ such that $\big||Q'| - \beta n\big| \le 100\eps n$ and all but at most $1200\eps n^2$ edges in $G[Q']$ are blue.
	\end{lemma}
	\begin{proof}
		Similarly to the proof of Theorem~\ref{theorem_intro_falgas_ravry}, start with $G' = G$, and, as long as $G'$ has a vertex $x$ which is not in a full blue clone-clique, run the symmetrisation procedure $S_{G'}(x)$, and replace $G'$ by the resulting graphs. Denote by $\Gf$ the graph at the end of the process (as before, the process is guaranteed to end). Then the vertices of $\Gf$ can be partitioned into full blue clone-cliques $Q_1, \ldots, Q_t$. 
		
		Let $Q'$ be the vertex set of the largest clone-clique. By \Cref{lem:theta}, we have $\big||Q'| - \beta n\big| \le 100\eps n$. In particular $|Q'| \in [0.55n, 0.65n]$.
		
		Let $F_1$ be the graph created just before the symmetrisation procedure was started on an element of $Q'$, and let $F_2$ be the graph just after $Q'$ was built. Notice that $c(F_2) \ge c(F_1) \ge c(G) \ge f(n) - \eps^2 n^3$.
		By \Cref{lem:Q-blue}, all but at most $1200\eps n^2$ edges in $F_1[Q']$ are blue. Notice that during the above process, the edges in $Q'$ remain untouched until right before a symmetrisation process is started on an element of $Q'$. It follows that all but at most $1200 \eps n^2$ edges in $G[Q']$ are blue.
	\end{proof}

	Now we can complete the proof by running a symmetrisation procedure in two phases. The first phase generates a blue clique $Q$ which contains almost all the vertices of $Q'$. The second phase allows us to show that $Q$ cannot be much larger than $Q'$ and to control the remaining edges incident to $Q$.
	
	\begin{proof}[Proof of Lemma~\ref{lem:ind-step}]
		Apply \Cref{lem:beta-clique} to find $Q'$ such that $\big||Q'| - \beta n\big| \le 100\eps n$ and $G[Q']$ has at most $\delta^2 n^2$ red edges (with $\delta^2= 1200\eps$).
		
		\begin{claim}
			We can run a symmetrisation procedure on $G$ which results in a graph $G'$ and a full blue clone-clique $Q$ satisfying $|Q' \setminus Q| \leq 3 \delta n$.
		\end{claim}
		\begin{proof}
			Let $A$ be the set of vertices in $Q'$ with more than $\delta n$ red (in- or out-) neighbours in $G[Q']$. The bound on the number of red edges in $Q'$ gives $|A| < 2\delta n$. Define $Q'' := Q' \setminus A$.
			
			We will run a symmetrisation procedure on $G$, but with a specific ordering of vertices. 
			We start with $x_1 \in Q''$ (chosen arbitrarily). Assuming that $\{x_1, \ldots, x_k\}$ are defined and contained in $Q''$, if possible we pick $x_{k+1}$ to also be in $Q''$ (we can do this as long as there is a vertex in $Q'' \setminus \{x_1, \ldots, x_k\}$ whose edges to $\{x_1, \ldots, x_k\}$ are blue). Once this is no longer possible, we continue with the symmetrisation procedure using an arbitrary order of vertices. Let $Q$ be the full blue clone-clique built by this procedure.
			
			Let $k$ be largest such that $\{x_1, \ldots, x_k\} \subseteq Q''$. It is easy to see that throughout the procedure, until at least step $k$, every vertex in $Q''$ has at most $\delta n$ non-blue neighbours in $Q'' \setminus \{x_1, \ldots, x_k\}$. Thus $k \ge |Q''| - \delta n \ge |Q| - 3\delta n$, as otherwise we could find a suitable $x_{k+1}$ in $Q''$, contradicting the choice of $k$. It follows that $|Q' \setminus Q| \le 3\delta n$.
		\end{proof}
		
		Let $G'$ and $Q$ be as in the above Claim. We claim that $|Q| \le (\beta + 100\eps)n$. Indeed, this follows from \Cref{lem:theta} by running symmetrisation procedures repeatedly, starting from $G'$, until the vertices can be partitioned into full blue clone-cliques (one of which is $Q$). It follows that $|Q \setminus Q'| \le 3\delta n + |Q| - |Q'| \le (3\delta + 200\eps)n \le 5\delta n$. In particular, the number of red edges in $G[Q]$ is at most the number of red edges in $G[Q']$ plus the number of edges incident with $Q \setminus Q'$, which amounts to a total of at most $(\delta^2 + 5\delta)n^2 \le 10\delta n^2$ red edges in $G[Q]$.
		
		By \Cref{lem:Q-red}, all but at most $10\eps n^2$ edges in $G[Q, V \setminus Q]$ are red and oriented towards $Q$, and similarly for $G'[Q, V \setminus Q]$.
		\def \Vi{V_{\inn}} 
		\def \Vo{V_{\out}} 
		
		Since $Q$ is a full blue clone-clique in $G'$, the vertices in $V \setminus Q$ can be partitioned into $\Vi$ and $\Vo$, where $vq$ is a red arc for every $v \in \Vi$ and $q \in Q$ and $qv$ is a red arc for $v \in \Vo$ and $q \in Q$. Thus, by the previous paragraph and because $|Q| \ge n/2$, $|\Vo| \le 20\eps n$.
		
		Let $G''$ be obtained from $G'$ be reorienting all $(Q, V \setminus Q)$ edges towards $Q$. Then
		\begin{equation*}
			c(G'') - c(G') 
			\ge \binom{|Q|}{2}{|\Vo|} - \binom{|\Vo|}{2}|Q|
			= |Q||\Vo| \cdot (|Q| - |\Vo|) \ge 0.
		\end{equation*}
		It follows that $c(G'') \ge c(G') \ge c(G)$. Moreover, $G''$ and $G'$ differ on at most $|\Vo|n \le 20\eps n^2$ edges, and thus $G$ and $G''$ differ on at most $(20\eps + 10\eps + 10\delta)n^2 \le 20\delta n^2$ edges.
		Since $G''$ has the required structure, this proves \Cref{lem:ind-step}.
	\end{proof}
	
	\subsection{Full stability result} \label{subsec:iteration}
	\begin{proof}[Proof of \Cref{thm:stability}]
		Let $\eps_1 \ll \eta \ll \eps_2$.
		The idea is simply to iterate \Cref{lem:ind-step}. 
		We will find graphs $G_1, \ldots, G_s$ and sets $Q_1, \ldots, Q_s$, satisfying the following conditions, for $k \in [s]$ (for convenience set $G_0 := G$, $Q_0 := \emptyset$ and $V := V(G)$).
		\begin{enumerate}[label = \rm(\arabic*)]
			\item \label{itm:full-stab-1}
			$G_k$ is a coloured graph on vertex set $V \setminus (Q_1 \cup \ldots \cup Q_{k-1})$.
			\item \label{itm:full-stab-2}
			$Q_k$ is a blue clique in $G_k$, all other edges incident with $Q_k$ in $G_k$ are red and point towards $Q_k$.
			\item \label{itm:full-stab-3}
			$\big||Q_k| - \beta|G_k|\big| \le \eta|G_k|$.
			\item \label{itm:full-stab-4}
			$G_k$ and $G_{k-1} \setminus Q_{k-1}$ differ on at most $\eta |G_{k}|^2$ edges.
			\item \label{itm:full-stab-5}
			$c(G_k) \ge c(G_{k-1} \setminus Q_{k-1})$.
			\item \label{itm:full-stab-6}
			$c(G_k \setminus Q_k) \ge f(|G_k \setminus Q_k|) - \eps_1 n^3$.
		\end{enumerate}
		To see how such a sequence can be built, suppose that $G_1, \ldots, G_{k-1}$ and $Q_1, \ldots, Q_{k-1}$ are defined and satisfy the above conditions. If $|G_{k-1} \setminus Q_{k-1}| \le \eta n$, we stop the process and set $s := k-1$. Otherwise, we apply \Cref{lem:ind-step} to the graph $G_{k-1} \setminus Q_{k-1}$. Notice that by \ref{itm:full-stab-6} and the assumption on $|G_{k-1} \setminus Q_{k-1}|$, we have $c(G_k \setminus Q_k) \ge f(|G_k \setminus Q_k) - \eps_1 \eta^{-3} |G_k \setminus Q_k|^3$. Since $\eps_1 \eta^{-3} \ll \eta$, the lemma is applicable. The lemma produces a graph $G_k$ on vertex set $V(G_{k-1}) \setminus Q_{k-1} = V \setminus (Q_1 \cup \ldots \cup Q_{k-1})$ satisfying items \ref{itm:full-stab-1} to \ref{itm:full-stab-5}. It remains to verify \ref{itm:full-stab-6}.
		Note that
		\begin{align*}
			c(G_k) = \binom{|Q_k|}{2} \cdot |G_k \setminus Q_k| + c(G_k \setminus Q_k).
		\end{align*}
		Also
		\begin{align*}
			c(G_k) 
			\ge c(G_{k-1} \setminus Q_{k-1}) 
			& \ge f(|G_{k-1} \setminus Q_{k-1}|) - \eps_1 n^3 \\
			& = f(|G_k|) - \eps_1 n^3 \\
			& \ge \binom{|Q_k|}{2}|G_k \setminus Q_k| + f(|G_k \setminus Q_k|) - \eps_1 n^3,
		\end{align*}
		where the last inequality follows from the definition of $f$. The two inequalities imply \ref{itm:full-stab-6}.
		
		To finish, run a symmetrisation procedure on $G_s \setminus Q_s$ repeatedly, to obtain a graph $H$ whose vertices are partitioned into full blue clone-cliques $Q_{s+1}, \ldots, Q_t$ (arranged in decreasing size); the edges between any two of them point towards the larger clique; and $c(H) \ge c(G_t \setminus Q_t)$. Let $G'$ be the graph on vertex set $V$, such that $Q_1, \ldots, Q_t$ are blue cliques and the edges between any two of them are red and point towards the larger clique (note that $Q_1, \ldots, Q_t$ partition $V$).
		
		To complete the proof of \Cref{thm:stability}, we need to show that properties \ref{itm:stability-a} to \ref{itm:stability-c} hold.	
		For \ref{itm:stability-a}, define $G_k'$ to be the graph on vertex set $V$, obtained from $G'$ by replacing $V \setminus (Q_1 \cup \ldots \cup Q_{k-1})$ by a copy of $G_k$ (this makes sense due to \ref{itm:full-stab-1}).
		It is easy to see that $c(G_k') - c(G_{k-1}') = c(G_k) - c(G_{k-1} \setminus Q_{k-1}) \ge 0$ for $k \in [s]$, using \ref{itm:full-stab-5}. Similarly, $c(G') \ge c(G_s')$. Altogether, $c(G') \ge c(G_1') = c(G)$, as required for \ref{itm:stability-a}.
		
		Before continuing, we derive an upper bound on $s$. By \ref{itm:full-stab-3} we have $|Q_k| \ge 0.55|G_k|$ for $k \in [s]$, so $|G_k| \le 2^{-(k-1)}n$. Since $|G_t| \le \eta n$, this implies that $s \le 2\log(1/\eta) \le \eta^{-1/2}$, say.
		
		By \ref{itm:full-stab-4} we find that $G'$ and $G$ differ on at most $((s \eta + \eta)n^2 \le 2\eta^{1/2}n^2 \le \eps_2 n^2$ edges. Property \ref{itm:stability-b} follows.
		
		Notice that the estimate $|Q_k| \ge 0.55|G_k|$, which follows from \ref{itm:full-stab-3} implies $|Q_1| \ge \ldots \ge |Q_s|$. Thus \ref{itm:stability-c1} to \ref{itm:stability-c3} clearly hold. Finally, \ref{itm:stability-c4} holds trivially for $k > s$ and, for $k \le s$, it follows from~\ref{itm:full-stab-3} and $\eta \leq \eps_2$.
		\end{proof}
		
		\section{Hypergraphs with no short odd pseudocycles} \label{sec:diameter}
		
		In this section we leverage the stability result about cherries, \Cref{thm:stability}, and the connection between hypergraphs with no odd pseudocycles to good colourings (\Cref{thm:good-colouring}) to prove the following result regarding the structure of a dense hypergraph with no short odd pseudocycles. In case of cycles and pseudocycles, the \textit{length} (number of edges) and order (number of vertices) coincide, so, since there is no danger of confusion, we prefer the term \textit{length}. Given vertex sets $X_1, X_2, X_3 \subset V(\HH)$, an \textit{$X_1X_2X_3$-triple} in $\HH$ is an (unordered) edge $x_1x_2x_3 \in E(\HH)$ with $x_i \in X_i$ for $i \in [3]$.
		
		\begin{theorem} \label{thm:partition}
			Let $n \gg \ell \gg 1$.
			Let $\HH$ be a $3$-uniform hypergraph on $n$ vertices which contains no odd pseudocycles of length at most $\ell$, and which maximises the number of edges under these conditions.
			Then there is a partition $\{A, B\}$ of the vertices of $\HH$ into non-empty sets such that all $AAB$ triples are edges of $\HH$ (and there are no $AAA$ and $ABB$ triples).
		\end{theorem} 
		
		By iterating the above result, we prove Theorem~\ref{thm:pseudocycles}, restated here, which gives an upper bound on the number of edges in a hypergraph with no short odd pseudocycles. 
		\thmPseudocycle*

		Recall that \Cref{thm:pseudocycles} is tight, up to the additive $O(1)$ error term, as evidenced by $\HH(x_1, \ldots, x_k)$ for a suitable choice of $x_i$'s. 
		
		We next show how \Cref{thm:pseudocycles} implies our main result, \Cref{thm:single-cycle}, restated here. 
		\thmSingleCycle*
		Recall that the \emph{$t$-blow-up} of an $r$-uniform hypergraph $\HH$, denoted $\HH[t]$, is the hypergraph with vertex set $V(\HH) \times [t]$ and edges all $r$-sets $\{(x_1, i_1), \ldots, (x_r, i_r)\}$ such that $\{x_1, \ldots, x_r\} \in E(\HH)$. 
		For a family $\Fc$ of hypergraphs, we denote by $\Fc[t]$ the family of $t$-blow-ups of members of $\Fc$. Recall that \Cref{theorem_blow_up} (whose proof can be found in \cite{keevash11}) asserts that taking the $t$-blow-up of a hypergraph does not change its Tur\'an density. The following generalisation for finite families of hypergraphs can be proved similarly.
		\begin{theorem}[\cite{keevash11}, Theorem 2.2] \label{thm:ex-blowup}
			Let $s$ and $t$ be integers, and let $\Fc$ be a family of $r$-graphs with $|\Fc|\leq s$. Then $\pi(\Fc[t]) = \pi(\Fc)$.
		\end{theorem}
		
		To prove \Cref{thm:single-cycle}, we will note that an odd cycle $C^{(3)}_m$ is contained in an \textit{$m$-blow-up} of any odd pseudocycle of length at most $m/2$, and apply the last theorem.
		
		\begin{proof}[Proof of Theorem~\ref{thm:single-cycle} using \Cref{thm:pseudocycles}]
			Let $m$ be an integer with $m \geq 2L$ and $3 \nmid m$, where $L$ is the constant from Theorem~\ref{thm:pseudocycles}. Recall that $f(n)  = (2\sqrt{3} - 3 + o(1)) \binom n3$. Let $\eps >0$ and let $\HH$ be an $n$-vertex 3-uniform hypergraph with 
			$e(\HH) \geq (2\sqrt{3} - 3 + \eps) \binom n3$ and $n$ sufficiently large. We claim that $\HH$ contains a copy of $C^{(3)}_m$.
			
			Theorem~\ref{thm:pseudocycles} and Theorem~\ref{thm:ex-blowup} imply that $\HH$ contains $F[m]$ for some $\ell$-pseudocycle $F$ with $\ell \leq L$ and $3 \nmid \ell$. It suffices to show that $F$ contains an $m$-pseudocycle, because then $C_m^{(3)}$ will be contained in $F[m]$. To see this, let $v_1 \dots v_\ell$  be an ordering of $V(F)$ such that $v_i v_{i+1} v_{i+2} \in E(F)$, with the indices taken modulo $\ell$.
			
			In case $m \modthree{\ell}$, consider the sequence $$(v_1 v_2 v_3) ^{\frac{m-\ell}{3}}v_1 v_2 \dots v_\ell,$$
			where $(v_1 v_2 v_3)^x$ stands for $x$ repetitions of the sequence $v_1 v_2 v_3$.
			This is a sequence of order $m$ certifying that $F$ contains an $m$-pseudocycle.
			
			Otherwise, if $m \modthree {2\ell}$, the same is certified for instance by the sequence
			\begin{equation*}
				(v_1 v_2 v_3) ^{\frac{m-2\ell}{3}} (v_1 v_2 \dots v_\ell) ^2. \qedhere 
			\end{equation*}
		\end{proof}
		
		All that remains now is to prove \Cref{thm:partition}. We will state and prove some preliminary results in the following subsection, and then prove the theorem in \Cref{subsec:proof-partition}.
		
		\subsection{Preparation}
		The \emph{diameter} of a hypergraph $\HH$ is the minimum $\ell$ such that the following holds: for every $x, y, z, w \in V(\HH)$ (where $x, y$ are distinct and $z, w$ are distinct) whenever there is a pseudopath from $xy$ to $zw$, there is such a pseudopath of order at most $\ell$.
		
		We have already shown that $n$-vertex hypergraphs with no odd pseudocycles have at most $f(n)$ edges. To prove the same for pseudocycles of bounded length, we will pass to a subhypergraph with bounded diameter, which is the purpose of the following two propositions.
		
		\begin{proposition} \label{prop:diam-cyc}
			Let $\HH$ be a 3-uniform hypergraph of diameter $\ell \geq 4$. If $\HH$ has an odd pseudocycle, then it has an odd pseudocycle of length at most $4 \ell$.
		\end{proposition}
		
		\begin{proof}
			Let $C$ be the shortest odd pseudocycle in $\HH$. Assuming that its length is at least $3\ell +4$, we may index it by $xy v_1 \ldots v_k ab u_1 \ldots u_t$ with $t \geq 2 \ell$, $k \geq \ell $. Note that the length of $C$ is $k+t+4 \not\equiv 0 \pmod{3}$.
			
			Since $\HH$ contains a pseudopath from $xy$ to $ab$, it also contains such a pseudopath  $P = xy w_1 \ldots w_r ab$ with $r \leq \ell -4$.	
			The pseudocycle $xy w_1 \ldots w_r ab u_1 \ldots u_t$ is shorter than $C$, so it must not be odd, that is, $r+t+4 \modthree{0}$.
			
			Now consider the pseudocycle $C_1 = v_1 \ldots v_k \tilde{P}$. Recall that $\tP$ is a $(2r+6)$-vertex pseudopath from $ab$ to $xy$ (see \eqref{eqn:tilde}), so $C_1$ is indeed a pseudocycle.
			The length of $C_1$ is $k + 2r +6 \equiv k-r \equiv k+t+4 \not\equiv 0 \pmod{3}$. Noting that $k+2r +6 \leq k + 2\ell - 2 \le k + t$, this contradicts the minimality of $C$.
		\end{proof} 
		
		\begin{proposition} \label{prop:small-diameter}
			Let $ 1/\ell \ll \eps \ll 1$, and let $\HH$ be an $n$-vertex hypergraph.
			Then there is a subgraph $\HH' \subseteq \HH$ with $e(\HH') \ge e(\HH) - \eps n^3$ whose diameter is at most $\ell$.
		\end{proposition}
		
		\begin{proof}
			First we form a subgraph $\HH' \subseteq \HH$ in which each vertex pair has codegree either $0$ or at least $\eps n$, as follows. If there are vertices $u, v$ whose codegree in the \emph{current} hypergraph is smaller than $\eps n$, delete all edges containing $uv$. Repeat this step until each pair has codegree degree either $0$ or at least $\eps n$. Denote the resulting hypergraph by $\HH'$. Observe that  the number of deleted edges is at most $\eps n \cdot \binom n2$ since the edges containing each pair were removed at most once. Hence $e(\HH') \geq e(\HH)-\eps n^3$.
			
			Given ordered pairs $uv$ and $u'v'$ which are connected by a pseudopath in $\HH'$, let $P=uv x_0 x_1 \dots x_t u'v'$ be a shortest such pseudopath. For each $i$, let $B_i$ be the set of ordered pairs $ab$ such that $x_ix_{i+1}ab$ is a tight path in $\HH'$. We claim that the sets $B_{10i}$ are mutually disjoint for $0\leq i < \frac{t}{10}$. Suppose not, and take $ab \in B_{10i} \cap B_{10j}$ for some $0 \le i<j < \frac{t}{10}$. Then $x_ix_{i+1}ab x_{j+1}a x_j x_{j+1}$ is a pseudopath  with only five vertices between $x_i$ and $x_j$, which can be used to form a shorter pseudopath  than $P$ connecting $uv$ and $u'v'$, contradiction.
			Now since $|B_i| \geq \eps^2 n^2/2$ for every $i$ (using the fact that the codegree of each pair in $\HH'$ is either 0 or at least $\eps n$), we have
			$$
			\floor{\frac{t}{10}}\cdot \frac{\eps^2 n^2}{2} \leq n^2,
			$$
			so $t \leq \frac{20}{\eps^2}$.
			Hence the diameter of $\HH'$ is at most $\ell:=\frac{20}{\eps^2}+4$, as required.
		\end{proof}
		
		As alluded to in \Cref{sec:overview}, we can already prove \Cref{cor:weak-main}, restated here, which is a weakening of \Cref{thm:pseudocycles}, with only an asymptotic upper bound, which depends on $\ell$, on the number of edges. 
		\corWeak*
		This bound will be used in the proof of \Cref{prop:max-deg}. Note that the analogous bound on the extremal number of proper odd tight cycles follows from \Cref{thm:ex-blowup}.
		
		\begin{proof}[Proof of \Cref{cor:weak-main}]
			Assume the opposite, that $e(\HH) \geq f(n) + \eps n^3$. Applying Proposition~\ref{prop:small-diameter} with the parameters $\ell/4$ and $\eps/2$, we obtain a hypergraph $\HH' \subseteq \HH$ with at least $f(n) + \eps n^3 /2$ edges whose diameter is at most $\ell/4$. $\HH'$ contains no odd pseudocycles of length at most $\ell$, so by \Cref{prop:diam-cyc}, it contains no odd pseudocycles. Hence we may apply Theorem~\ref{thm:good-colouring} to obtain a good colouring of $\partial \HH'$ with $e(\HH') > f(n)$ cherries, contradicting \Cref{theorem_intro_falgas_ravry}.
		\end{proof}
		
		The following proposition gives a near-optimal lower bound on the vertex degrees in a largest hypergraph on $n$ vertices with no short odd pseudocycles.
		
		\begin{proposition} \label{prop:min-deg} 
			Let $1/n \ll 1/\ell \ll \eps \ll 1$, and let $\HH$ be an $n$-vertex hypergraph with no odd pseudocycles of length at most $\ell$, which maximises the number of edges under these conditions. Then $d(u) \ge (3\alpha - \eps)n^2$ for every vertex $u$.
		\end{proposition}
		
		\begin{proof}
			Given vertices $u$ and $v$ in $\HH$, consider the hypergraph $\HH_{uv}$ obtained from $\HH$ by removing all edges containing $v$ and then adding the edge $e - u + v$, for each edge $e$ that contains $u$ but not $v$. Observe that $\HH$ has no odd pseudocycles of length at most $\ell$; indeed, if there were such a cycle then we could replace each instance of $v$ by $u$ to obtain an odd pseudocycle of the same length in $\HH$ (whereby it is important that $\HH_{uv}$ has no edges containing both $u$ and $v$), a contradiction. Since $e(\HH_{uv}) \ge e(\HH) - d(v) + d(u) - n$ and by maximality of $\HH$, we have $d(v) \ge d(u) - n$. Since $u$ and $v$ were arbitrary, this implies that the maximum and minimum degrees of $\HH$ differ by at most $n$. In particular, using $e(\HH) \ge f(n) = \alpha n^3 + o(n^3)$, which follows from the maximality of $\HH$ and \Cref{prop:alpha},
			\begin{equation*}
				\delta(\HH) 
				\ge \frac{3e(\HH)}{n} - n \ge \frac{3f(n)}{n} - n \ge (3 \alpha - \eps)n^2.
				\qedhere
			\end{equation*}
		\end{proof}
		
		Next, we prove a stability version of the previous proposition.
		
		\begin{proposition} \label{prop:max-deg}
			Let $1/n \ll 1/\ell \ll \eps_1 \ll \eps_2 \ll 1$, and let $\HH$ be an $n$-vertex $3$-uniform hypergraph with no odd pseudocycles of length at most $\ell$. If $e(\HH) \ge f(n) - \eps_1 n^3$ then $d(u) \le (3\alpha + \eps_2)n^2$ for every vertex $u$.
		\end{proposition}
		
		\begin{proof}
			Let $\mu = \sqrt{\eps_1} \leq \eps_2/10$.
			Let $X$ be the set of vertices $x$ with $d(x) \le 3(\alpha + \mu)n^2$. Then $e(\HH) \ge (n - |X|)(\alpha + \mu)n^2$. By \Cref{cor:weak-main} (and the properties of $f(n)$) we also have $e(\HH) \le (\alpha + \eps_1)n^3$.
			Putting the two inequalities together, we get
			\begin{align*}
				& (\alpha + \eps_1)n^3 \ge (n - |X|)(\alpha + \mu)n^2 \\
				\Longrightarrow \quad
				& |X| \ge \frac{(\alpha + \mu)n - (\alpha + \eps_1)n}{\alpha + \mu} = \frac{\mu - \eps_1}{\alpha + \mu} \cdot n \ge \mu n.
			\end{align*}
			Let $u$ be a vertex of maximum degree in $\HH$, and let $X'$ be a subset of $X$ of size $t := \mu n$. We may assume $u \notin X'$ because otherwise $d_{\HH}(u) \le (3\alpha + 3\mu)n^2 \le (3\alpha + \eps_2)n^2$, as required.
			Now consider the hypergraph $\HH_1$ formed in two steps as follows. First, define $\HH_0 = \HH \setminus X'$; then $e(\HH_0) \ge e(\HH) - t \cdot 3(\alpha + \mu)n^2$ and $d_{\HH_0}(u) \ge d_{\HH}(u) - t n$. Second, let $\HH_1$ be the hypergraph obtained by adding $|X'|$ copies of $u$ to $\HH_0$. Then 
			\begin{align*}
				e(\HH_1) 
				& \ge e(\HH_0) + t \cdot d_{\HH_0}(u) \\
				& \ge e(\HH) - t \cdot 3(\alpha + \mu)n^2 + t \cdot (d_{\HH}(u) - tn) \\
				& \ge f(n) - \eps_1 n^3 + t \cdot (d_{\HH}(u) - tn - 3(\alpha + \mu)n^2) \\
				& = f(n) - \eps_1 n^3 + \mu n \cdot (d_{\HH}(u) - (3\alpha + 4\mu)n^2).
			\end{align*}
			Notice that $\HH_1$ has no odd pseudocycles of length at most $\ell$. Thus, by \Cref{cor:weak-main}, we have $e(\HH_1) \le f(n) + \eps_1 n^3$. Hence, using $\mu = \sqrt{\eps_1} \leq \eps_2 / 10$,
			\begin{equation*}
				d_{\HH}(u) 
				\le (3\alpha + 4\mu)n^2 + (2\eps_1/\mu)n^2
				\le (3\alpha + \eps_2)n^2,
			\end{equation*}
			as required. 
		\end{proof}

		\subsection{The structure of odd-pseudocycle-free graphs} \label{subsec:proof-partition}
		We now prove the main result in the section, \Cref{thm:partition}.
		The starting point of the proof uses the relation between hypergraphs with no odd pseudocycles and good colourings of $K_n$, as well as the stability result about cherries from the previous section, to conclude the following: there is a coloured graph $G$ with a nice structure such that almost all cherries in $G$ are triples in $\HH$ and vice versa. This readily implies the existence of a partition $\{A, B\}$ of the vertices such that $|A| \approx \beta n$ and for almost every vertex $u$ in $\HH$ the following holds: almost all vertices in $A$ are joined to almost all $A \times B$ pairs, and almost all vertices in $B$ are joined to almost all $A^{(2)}$ pairs. The main difficulty of the proof lies in showing that there is such a partition for which \textit{every} vertex in $A$ is joined to almost all pairs in $A \times B$, and similarly for vertices in $B$. This is achieved in \Cref{claim:structure} and the main idea is to compare several graphs obtained by modifying the triples containing a given vertex. Given a partition as above, to conclude the proof, we argue (using the fact that $\HH$ has no short odd pseudocycles) that the number of $AAB$ ``non-edges'' exceeds the number of $AAA$ and $ABB$ edges, unless all of these numbers are 0. The maximality of $\HH$ implies that all these numbers are indeed $0$, meaning that $\HH$ has all $AAB$ edges and no $AAA$, $ABB$ edges.
		
		\begin{proof}[Proof of \Cref{thm:partition}]
			Let $\eps_7 = 0.1$ and let $\eps_1, \ldots, \eps_6$, and $\ell$ satisfy 
			\begin{equation*}
				0 < 1/\ell \ll \eps_1 \ll \ldots \ll \eps_7. 
			\end{equation*}
			Let $\HH'$ be a subgraph of $\HH$ on the same vertex set with at least $e(\HH) - \eps_1 n^3$ edges, that has diameter at most $\ell/4$; such $\HH'$ exists by \Cref{prop:small-diameter}. By \Cref{prop:diam-cyc},  $\HH'$ has no odd pseudocycles, so by \Cref{thm:good-colouring}, there is a good colouring of $\partial\HH'$. 
			
			Extending the good colouring of $\partial\HH'$ arbitrarily to also cover vertex pairs which are not in the shadow, we obtain a coloured graph (recall that this is a complete graph whose edges are either blue or oriented and red) $G'$ on vertex set $V := V(\HH)$, such that every edge in $\HH'$ is a cherry in $G'$.
			By maximality of $\HH$, we have $c(G') \ge e(\HH') \ge e(\HH) - \eps_1 n^3 \ge f(n) - \eps_1 n^3$.
			
			Thus, by \Cref{thm:stability}, there is a graph $G$ satisfying \ref{itm:stability-a}--\ref{itm:stability-c} in \Cref{thm:stability} on vertex set $V$. That is, $G$ has at least as many cherries as $G'$, all but at most $\eps_2 n^3$ cherries in $G$ are cherries in $G'$, and $V$ can be partitioned into sets $X_1, \ldots, X_k$ such that: $G[X_i]$ is blue for $i \in [k]$; $|X_i| = (\beta \pm \eps_2 )n \cdot (|X_{i}| + \ldots + |X_k|)$ for $i \in [k]$; and all $X_i \times X_j$ pairs in $G$ are red and oriented towards $X_i$, for $1 \le i < j \le k$. Recall that $\beta=\frac{3 - \sqrt{3}}{2}$ was defined in \eqref{eq:alphabeta}.
			
			Define $X_{>i} := X_{i+1} \cup \ldots \cup X_k$, and define $X_{\ge i}$ analogously.
			Let $H$ be the subgraph of $G$ whose edges are either pairs in $X_i \times X_i$ that are in at least $(|X_{i+1}| + \ldots + |X_k|) - \eps_3 n$ triples in $(X_i \times X_i \times X_{>i}) \cap E(\HH)$, or pairs in $X_i \times X_j$, where $i < j$, that are in at least $|X_i| - \eps_3 n$ triples in $(X_i \times X_i \times X_j) \cap E(\HH)$. 
			
			Denoting the number of non-edges in $H$ by $\be(H)$, we have that the number of cherries in $G$ that are not edges in $\HH$ is at least $\be(H) \cdot \eps_3 n / 3$. 
			Recall that $e(\HH') \ge e(\HH) - \eps_1 n^3 \ge f(n) - \eps_1 n^3$ and that all edges in $\HH'$ are cherries in $G'$. But $c(G') \leq f(n)$ (by~Theorem~\ref{theorem_intro_falgas_ravry}), so all but $\eps_1 n^3$ cherries in $G'$ are edges in $\HH'$ and thus in $\HH$. Since there are at most $\eps_2 n^3$ cherries in $G$ that are not cherries in $G'$, it follows that all but at most $(\eps_1 + \eps_2)n^3 \le 2\eps_2 n^3$ cherries in $G$ are edges in $\HH$.
			Hence $\be(H) \cdot \eps_3 n / 3 \le 2\eps_2 n^3$, showing $\be(H) \le (6\eps_2/ \eps_3)n^3 \le \eps_3 n^2$.
			
			Let $k_0$ be the maximum $i$ such that $|X_i| \ge \eps_4 n$.
			Define subsets $X_i' \subseteq X_i$ as follows: if $i < k_0$ let $X_i'$ be the set of vertices in $X_i$ that have degree at least $|X_i| - \eps_4 n$ in $H[X_i]$ and degree at least $|X_{>i}| - \eps_4 n$ in $H[X_i, X_{>i}]$; if $i \ge k_0$, define $X_i' := \emptyset$.
			Since $ (\eps_4 n / 2)\sum_{i < k_0} |X_i \setminus X_i'| \le \be(H) \le \eps_3 n^2$ and $|X_{\ge k_0}| \le 10 \eps_4 n$ (using \ref{itm:stability-c4}), we have $$\sum_{i \in [k]} |X_i \setminus X_i'| \le 10\eps_4 n + (2\eps_3/\eps_4)n \le 20 \eps_4n.$$
			Let $X := X_1' \cup \ldots \cup X_k'$ and $Y := V \setminus X$. We have seen that $|Y| \le 20\eps_4 n \le \eps_5 n$.

			For $v \in V$, let $N(v)$ be the \emph{link} of $v$, namely the graph spanned by pairs $uw$ such that $uvw \in E(\HH)$.
			Write $A := X_1'$ and $B := X \setminus X_1'$.
			\begin{claim} \label{claim:structure}
				One of the graphs $N(u)[A]$ and $N(u)[A, B]$ has at most $\eps_6 n^2$ non-edges, for every $u \in V$.
			\end{claim}
			
			\begin{proof}
				Let $\eps_5 \ll \mu \ll \eps_6$. Note that the claim holds for all $u \in X$, so it suffices to prove it for $u \in Y$. Fix such $u$.
				
				Let $\F$ be the hypergraph on vertex set $X$ whose edges are all $X_i'X_i'X_j'$ triples with $1 \le i < j \le k_0$. We will construct two hypergraphs $\F_i^{+}$ (for $i \in \{1, 2\}$), that consist of $\F$ with one additional vertex $u_i$, which is a suitable modification of $u$, and that have no odd pseudocycles of length at most $\ell / 10$. We will argue that if both $N(u)[A]$ and $N(u)[A, B]$ have at least $\eps_6 n^2$ non-edges then $d_{\F_i^+}(u_i) >(3\alpha + \mu)n^2$ for some $i \in [2]$, contradicting \Cref{prop:max-deg}.
				
				Let $F_0$ be the graph on vertex set $X$ with edges $E(H) \cap E(N(u))$. Recall that vertices in $X_i'$ have at most $2\eps_4 n$ non-neighbours in $H[X'_{>i}]$. Thus,  using \Cref{prop:min-deg} for a lower bound on $d_{\HH}(u)$, we have $e(F_0) \ge d_{\HH}(u) - |Y| \cdot n - |X| \cdot 2\eps_4 n \ge (3\alpha - 10\eps_5)n^2$.
				We modify $F_0$ as follows, while possible: remove each edge $xy$ satisfying: $x, y \in A$ and $x$ has degree $1$ in $A$; or $x \in A$, $y \in B$, and $x$ has degree $1$ into $B$ or $y$ has degree $1$ into $A$. Call the resulting graph $F$ and notice that $|E(F_0) \setminus E(F)| \le 2n$, implying that 
				\begin{equation} \label{eqn:deg-u}
					e(F) \ge (3\alpha - 20 \eps_5)n^2.
				\end{equation}
				
				Recall that $\F$ is the hypergraph on vertex set $X$ whose edges are all $X_i'X_i'X_j'$ triples with $1 \le i < j \le k_0$, and let $\Fp$ be the hypergraph obtained by adding the vertex $u$ to $\F$ along with all edges $uvw$ such that $vw \in E(F)$.
				We argue that $\Fp$ has no odd pseudocycles of length at most $\ell/10$. To do so, we prove the following.
				\begin{align} \label{eqn:F-to-H}
					\begin{split}
						&\text{Let $xy, vw \in E(H)$, and let $P$ be a pseudopath  in $\F$ from $xy$ to $vw$ on $t$ vertices. Then}  \\
						&\text{there is a pseudopath $P'$ in $\HH$ from $xy$ to $vw$ of order $t$ (if $t \in \{2, 3\}$) or $t + 3$ (otherwise).}
					\end{split}
				\end{align}
				
				We prove \eqref{eqn:F-to-H} by induction on $t$. If $t = 2$ we can take $P' = P$. Suppose that $t = 3$, so $P = xyw$. Let $i_1, i_2, i_3$ be such that $x \in X_{i_1}'$, $y \in X_{i_2}'$ and $w \in X_{i_3}'$.
				Since $xy \in E(H)$, we know that for almost every $a \in X_{i_3}'$ the following holds: $xya \in E(\HH)$ and $ya \in E(H)$; pick such an $a$ with $a \neq w$. Similarly, $yab$ and $ywb$ are edges in $\HH$ for almost every $b \in X_{i_1}'$; pick such $b$. The path $xyabyw$ satisfies the requirements.
				
				Next, suppose that $t = 4$, so $P = xyvw$. Let $i_1, i_2, i_3, i_4$ be such that $x \in X_{i_1}'$, $y \in X_{i_2}'$, $v \in X_{i_3}'$ and $w \in X_{i_4}'$. As $xy \in E(H)$, almost all $a \in X_{i_3}'$ satisfy $xya \in E(\HH)$ and $ya \in E(H)$; fix such $a$. Similarly, almost all $c \in X_{i_2}'$ satisfy $cvw \in E(\HH)$, $cv \in E(H)$ and $ac \in E(H)$; fix such $c$. Finally, almost every $b \in X_{i_3}'$ satisfies $yab, abc, bcv \in E(\HH)$; fix such $b$. Then $P' = xyabcvw$ satisfies the requirements.
				
				Finally, suppose that $t \ge 5$, and write $P = v_1 \ldots v_t$, so $x = v_1$, $y = v_2$, $v = v_{t-1}$ and $w = v_t$. Let $i_j$ be such that $v_j \in X_{i_j}'$ for $j \in [t]$. As usual, since $xy = v_1 v_2 \in E(H)$, almost all $a \in X_{i_3}'$ satisfy: $v_2 a \in E(H)$ and $v_1 v_2 a \in E(\HH)$. Let $Q = v_2 a v_4 \ldots v_t$. Then $Q$ is a pseudopath in $\F$ of order $t-1$ that starts and ends with edges in $H$. By induction, there is a pseudopath $Q'$ in $\HH$ from $v_2 a$ to $v_{t-1} v_t$ of order $t + 2$. Then we can take $P' = v_1 Q'$, completing the proof of \eqref{eqn:F-to-H}.
				
				Now suppose that $C = v_1 \ldots v_t$ is a pseudocycle in $\Fp$, where $t \le \ell/10$. We need to show that $t$ is divisible by $3$. If $C$ does not go through $u$, then $C$ is in $\F$, implying that $t$ is indeed divisible by $3$. So we may assume that $C$ goes through $u$ at least once. This shows that $C$ can be written as $u P_1 u \ldots u P_k$, where $P_i$ is a pseudopath in $\F$ whose first two vertices and last two vertices form edges in $F$. It follows from \eqref{eqn:F-to-H} that for each $i \in [k]$ there is a pseudopath $P_i'$ in $\HH$ whose first two vertices and last two vertices match those of $P_i$ and whose order satisfies $|P_i'| - |P_i| \in \{0, 3\}$. Then $C' := u P_1' u \ldots u P_k'$ is a cycle in $\HH$ with $|C'| \le |C| + 3k \le 4|C| \le \ell$ and $|C'| \modthree{|C|}$. By the properties of $\HH$, we have that $|C'|$ is divisible by $3$, implying that $|C|$ is divisible by $3$, as required.
				
				Let $A_0$ and $A_1$ be the sets of vertices in $A$ incident with $AA$ and $AB$ edges in $F$, respectively (that is, $a_0 \in A_0$ if $F$ contains an edge $a_0x$ with $x \in A$). 
				To show that $A_0$ and $A_1$ are disjoint, assume that $a_1 \in A_0 \cap A_1$, so that there is a path $a_0 a_1 b_0$ in $F$ with $a_0, a_1 \in A$ and $b_0 \in B$. By construction of $F$, $b_0$ has an $F$-neighbour $a_2 \in A - A_1$, so  $a_0a_1b_0a_2$ is a path in $F$. Let $a_3$ and $b_1$ be arbitrary vertices in $A$ and $B$, respectively (distinct from previously chosen vertices). Then $a_0 a_1 u b_0 a_2 a_3 b_1$ is cycle of length 7 in $\Fp$, a contradiction.
				
				Let $B_1$ be the set of vertices in $B$ incident with $AB$ edges in $F$. We  claim that $B_1$ is independent in $F$. Indeed, otherwise there is a path $a_1b_1b_2a_2$ in $F$, using a similar argument to the above paragraph. Now, choosing $a_3, a_4 \in A$ and  $b_3 \in B$ to be arbitrary unused vertices, we obtain a cycle $a_1b_1ub_2a_2 a_3 b_3 a_4$ of length 8 in $\Fp$ and reach a contradiction.
				
				Let $F_1$ and $F_2$ be graphs on vertex set $X$, defined as follows: $E(F_1) = A \times B$ and $E(F_2) = A^{(2)} \cup E(F[B])$. Now define $\Fp_i$ to be the graph obtained from $\F$ by adding a new vertex $u_i$ and edges $u_i vw$ such that $vw \in E(F_i)$, for $i \in [2]$. Thus $\Fp_i$ and $\Fp$ differ only on edges touching $u_i$ or $u$.
				We claim that $\Fp_i$ has no odd pseudocycles of length at most $\ell/10$. Indeed, this is easy to see for $i = 1$, because we can think of $\Fp_1$ as obtained by extending $X_1'$ by one vertex.  To see that this also holds for $i = 2$, notice that in $\Fp_2$, the $AAB$ and $BBB$ triples are in different strong components, so any pseudocycle  $C$ in $\Fp_2$ is either a pseudocycle in $\Fp$ or consists only of edges containing exactly two vertices from $A$. 
				
				Notice that $e(\Fp_i) \ge c(G) - |Y|n^2 \ge f(n) - (\eps_1 + \eps_5) n^3 \ge f(n) - 2\eps_5 n^3$, because all cherries in $G$ that do not touch $Y$ are edges in $\F$ and $c(G) \ge c(G') \ge f(n) - \eps_1 n^3$. Using this lower bound and the fact that $\Fp_i$ has no odd pseudocycles of length at most $\ell/10$, \Cref{prop:max-deg} implies that $d_{\Fp_i}(u_i) \le (3\alpha + \mu)n^2$. 
				Since $d_{\Fp}(u) = e(F) \ge (3\alpha - 20\eps_5)n^2$ (see \eqref{eqn:deg-u}), we have $e(F_i) - e(F) = d_{\Fp_i}(u_i)-d_{\Fp}(u) \le (\mu + 20\eps_5)n^2 \le 2\mu n^2$ for $i \in [2]$. 
				
				To finish, suppose first that $|A_0| \ge |B_1|$. Recalling that $F$ and $F_1$ coincide on $B$, and that $F$ has no edges in $(A_1 \cup B_1) \times A_0$ or $A_1^{(2)}$, we have
				\begin{align*}
					2\mu n^2 \ge e(F_2)-e(F) 
					& \ge - |A_1||B_1| + |A_0||A_1| + \binom{|A_1|}{2} + \be(F[A_0]) \\
					& \ge \frac{|A_1|^2}{2} + \be(F[A_0]) + O(n). 
				\end{align*}
				It follows that $|A_1| \le 5\mu^{1/2} n$ and $\be(F[A_0]) \le 5\mu n^2$. Altogether $\be(F[A]) \le |A_1| \, n + \be(F[A_0]) \le 10 \mu^{1/2} n^2 \le \eps_6 n^2$. Since $F[A] \subseteq N(u)[A]$, \Cref{claim:structure} is proved in this case.
				
				Now we consider the remaining case, namely that $|A_0| \le |B_1|$. Let $B_0 = B \setminus B_1$, and recall that $F$ has no edges in $B_1^{(2)}$ or in $A_0 \times B_1$. Using $|A| \ge |B| = |B_0| + |B_1|$,
				\begin{align*}
					2\mu n^2 & \ge e(\Fp_1) - e(\Fp) \\
					& \ge -\binom{|A_0|}{2} - \binom{|B_0|}{2} - |B_0||B_1| + |A||B_0| + |A_0||B_1| + \bar{e}(F[A_1, B_1]) \\
					& \ge |A_0|(|B_1| - |A_0|) + |B_0|(|A| - |B_0| - |B_1|) + \frac{|A_0|^2}{2} + \frac{|B_0|^2}{2} + \be(F[A_1, B_1]) + O(n) \\
					& \ge \frac{|A_0|^2}{2} + \frac{|B_0|^2}{2} + \be(F[A_1, B_1]) + O(n).
				\end{align*}
				Thus, we have $|A_0|, |B_0| \le 5 \mu^{1/2} n$ and $\be(F[A_1, B_1]) \le 5\mu n^2$. 
				This implies that $\be(F[A, B]) \le |A_0|\, n + |B_0|\, n + \be(F[A_1, B_1]) \le \eps_6 n^2$, proving \Cref{claim:structure}.
			\end{proof}
			
			Let $\As$ be the set of vertices $u$ such that $N(u)[A, B]$ has at most $\eps_6 n^2$ non-edges, and let $\Bs := V \setminus \As$. Note that $A \subseteq \As$, and by \Cref{claim:structure}, for every $u \in \Bs$ the graph $N(u)[A]$ has at most $\eps_6 n^2$ non-edges.
			Let $t_1$ be the number of $\As\As\As$ triples in $\HH$, let $t_2$ be the number of $\As\Bs\Bs$ triples in $\HH$, and let $s$ be the number of $\As\As\Bs$ triples that are not edges in $\HH$.
			Let $\Hs$ be the hypergraph obtained from $\HH$ by removing all $\As\As\As$ and $\As\Bs\Bs$ triples and adding all missing $\As\As\Bs$ triples. Then $\Hs$ has no odd pseudocycle of length at most $\ell$; this follows from observing that every pseudocycle in $\Hs$ is either a pseudocycle in $\HH$ or each of its edges has exactly two vertices in $\As$. Moreover, $e(\Hs) - e(\HH) = s - (t_1 + t_2)$. By maximality of $\HH$ we have $s \le t_1 + t_2$.
			
			\begin{claim}
				$t_1 \le \eps_7 s$.
			\end{claim}
			
			\begin{proof}
				Let $\eps_6 \ll \mu \ll \eps_7$.
				
				We first show that for every distinct $u, v \in \As$, there are at most $\mu n$ vertices $w \in \As$ such that $uvw \in E(\HH)$.
				
				Suppose there exist $u, v \in \As$ violating this. Let $W$ be the set of vertices $w \in \As$ such that $uvw \in E(\HH)$, so $|W|\geq \mu n$. 
				Consider the graph $(N(u) \cap N(v))[W, B]$; its edges are pairs $wb$ such that $w \in W$, $b \in B$, and $uwb, vwb \in E(\HH)$. This graph has at most $2\eps_6 n^2$ non-edges, by \Cref{claim:structure}. Thus there exists $b \in B$ with at least $\frac 12 \mu n$ neighbours in the aforementioned graph; denote its set of neighbours by $W'$. Now, by \Cref{claim:structure}, $b$ is adjacent in $\HH$ to all but at most  $\eps_6 n^2$ pairs in $W'$, so there exists a triple $w_1w_2b \in E(\HH)$ with $w_1, w_2 \in W'$. Thus $uvw_1bw_2$ is a pseudocycle of length 5, contradiction.

				To finish the argument, we count the four-tuples
				\begin{equation*}
					Q := \{ \{u, v, w, z \}: u, v, w \in \As, z \in \Bs, uvw \in E(\HH), uvz \notin E(\HH)\}
				\end{equation*}
				in two different ways. For each vertex $b \in \Bs$ and $\As\As\As$ triple $uvw \in E(\HH)$, at least one of the triples $uvb, uwb, vwb$ is not in $E(\HH)$ (since otherwise $\HH$ has a 4-cycle), so $|Q| \geq t_1|\Bs|$. On the other hand, it follows from the above paragraph that any $\As\As\Bs$ triple $uvz \notin E(\HH)$ extends to at most $\mu n$ elements of $Q$, so $|Q| \leq s \mu n$. Hence 
				\begin{equation*}
					t_1 \leq \frac{|Q|}{|\Bs|} \le \frac{s \mu n}{|\Bs|} \leq \eps_7 s,
				\end{equation*}
				as claimed.
			\end{proof}
			
			\begin{claim}
				$t_2 \le 2s/3$. 
			\end{claim}
			
			\begin{proof}
				Let $\eps_6 \ll \mu \ll \eps_7$.
				
				To begin with, we show that if $uvw$ is an $\As \Bs \Bs$ triple in $\HH$ (with $u \in \As$) then one of the pairs $uv$ and $vw$ is in at most $\mu n$ triples of form $\As\As\Bs$ in $\HH$. Fix an $\As\As\Bs$ triple $uvw \in E(\HH)$.
				
				Let $W'$ (resp.\ $V'$) be the set of vertices $a \in \As$ such that $uwa \in E(\HH)$ (resp.\ $uva \in E(\HH)$). Suppose that $|W'|, |V'| \geq \mu n$. Consider the graph $(N(v) \cap N(w))[W', V']$. By~\Cref{claim:structure}, this graph contains an edge $a_1a_2$, i.e.\ we have~$a_1a_2w$, $a_1a_2v \in E(\HH)$.  By definition of $W'$ and $V'$, the triples $uwa_1$ and $uva_2$ are in $\HH$. Hence $uwa_1a_2v$ is a cycle of length 5, contradiction.
				
				Let $F$ be an auxiliary bipartite graph with parts $\As$ and $\Bs$ such that $uv$ is an edge of $F$ whenever (i) there is an $\As\Bs\Bs$ triple in $\HH$ containing $uv$, and (ii) the number of $\As\As\Bs$ triples containing $uv$ is at most $\mu n$. By the previous paragraph, each $\As\Bs\Bs$ triple in $\HH$ contains an edge of $F$, so
				\begin{equation*}
					t_2 \leq |\Bs| \cdot e(F) \leq 0.4n \cdot e(F).
				\end{equation*}
				Moreover, we claim that $d_F(v) \le \mu n$ for every $v \in \Bs$. Indeed, by (ii), the graph $N(v)[\As]$ has at least $d_F(v)(|\As| - \mu n)/2$ non-edges. If $d_F(v)>\mu n$ then this quantity is larger than $2\eps_6 n^2$, contradicting \Cref{claim:structure}. Also using (ii), we conclude that
				\begin{equation*}
					s \ge \sum_{v \in \Bs} d_F(v) \cdot (|\As| - d_F(v) - \mu n) \ge 0.6 n \cdot  e(F).
				\end{equation*}
				It follows that $t_2 \le 2s/3$, as claimed.
			\end{proof}
			
			The last two claims, and the choice $\eps_7=0.1$, say, show that $(t_1 + t_2) \le 0.8 s$. Since $s \le t_1 + t_2$ this implies that $t_1 = t_2 = s = 0$. That is, all $\As\As\Bs$ triples are edges in $\HH$ (and there are no $\As\As\As$ or $\As\Bs\Bs$ edges). This proves \Cref{thm:partition}.
		\end{proof}
		
		\section{Open problems}
		There are several natural extensions of our result. Firstly, one could prove~\Cref{Conjecture_Mubayi_Rodl}, or perhaps determine the density of $C_\ell^{(3)}$ for smaller values of $\ell$, say $\ell \leq 100$. Although we do not state our bound on $\ell$ explicitly, this would not be too cumbersome, since it is a polynomial in $\eps_7$, and we set $\eps_7 = 0.1$.
		
		Of course our result should not extend to \emph{all} values of $\ell\equiv 1$ or $2 \pmod 3$, since for $\ell=4$, the tight cycle $C_{4}^{(3)}$ is the same as the tetrahedron $K_{4}^{(3)}$. Here the famous conjecture of Tur\'an says that $\pi(K_{4}^{(3)})=5/9$, which is attained by a wide family of extremal constructions~\cite{brown83,kostochka84,flaass88,razborov11}. Curiously, Fon-Der-Flaass showed that the conjectured extremal constructions $K_{4}^{(3)}$-free graphs can be constructed from oriented graphs in a manner reminiscent of Definition~\ref{def:good-col}. Specifically Fon-Der-Flass~\cite{flaass88} showed that if $D$ is an oriented graph with no induced directed 4-cycles, then the 3-graph formed by induced copies of $\{ab, ac\}$ and $\{ab,bc,ca\}$ will be $K_4^{(3)}$-free. 
		
		A second interesting direction is determining the Tur\'an density of $r$-uniform tight cycles for $r \geq 4$. For this, we do not even know of a conjectured optimal construction. Moreover, our characterisation of odd-pseudocycle-free hypergraphs (\Cref{thm:good-colouring}) does not have an obvious extension, as the straightforward extension of Definition~\ref{def:good-col} is too strong.

		Recall that in \Cref{thm:pseudocycles} we prove an almost tight result (up to a constant additive error) for the Tur\'an number of the family of pseudocycles of length $\ell$, for all $\ell \le L$ which are not divisible by $3$, and large enough $L$. It is plausible that the same could be proved for $\C_{\ell}$ for large enough $\ell$ which is not divisible by $3$. Namely, it is likely that $\ex(n, \C_{\ell}) \le f(n) + O(1)$ for such $\ell$? To tackle this, one is likely to require stability arguments, perhaps like those we used in \Cref{sec:diameter}. We remark that Liu, Mubayi, and Reiher \cite{liu2023unified} present a unified framework for tackling stability problems for a large class of hypergraph families. Unfortunately, this does not seem to be applicable in our case.
		
		As mentioned in the introduction, there are many other specific 3-uniform hypergraphs for which determining the Tur\'an density would be very interesting. Let us point out one conjecture which is perhaps less well known, and which can be found for instance in \cite{mubayi2011hypergraph}.
		\begin{conjecture} ~\label{conj:c5-minus}
			Let $\C_5^{-}$ be the 3-uniform hypergraph obtained from the tight 5-cycle $\C_5^3$ by removing one edge. The Tur\'an density of $\C_5^{-}$ is $\frac 14$.
		\end{conjecture}
		As in our case, one conjectured extremal hypergraph is an iterated construction; one may take a complete 3-partite 3-uniform hypergraph and then repeat the same construction recursively within each of the three parts. 
		Similarly to our result, Balogh and Haoran \cite{balogh2023turan} recently proved that the Tur\'an density of $\C_{\ell}^-$, for sufficiently large $\ell$ which is not divisible by $3$, is $\frac{1}{4}$.
		
		\subsection*{Acknowledgements}
		We would like to thank Jozsef Balogh, Xizhi Liu, Dhruv Mubayi, Yuejian Peng, Oleg Pikhurko, and Alexander Sidorenko for bringing to our attention several important references. We would also like to thank the anonymous referee for their helpful comments.
		
		\bibliography{refs-Turan}
		\bibliographystyle{amsplain}
	\end{document}